\documentclass[11pt, a4paper, reqno]{amsart}
\usepackage[textwidth=125mm,textheight=195mm]{geometry}
\usepackage{tikz}
\usepackage{amsmath,bm,bbm,amsthm, amssymb}
\usepackage{color}
\usepackage{animate}
\usepackage{etoolbox}
\usepackage{ulem}
\usepackage{comment}
\usepackage{pgf}
\usetikzlibrary{calc}
\usetikzlibrary{patterns}
\usetikzlibrary{arrows}
\usetikzlibrary{decorations.pathreplacing}
\usepackage[utf8]{inputenc}
\usepackage{pgfplots}
\theoremstyle{plain}
\newtheorem{theorem}{Theorem}[section]
\newtheorem{proposition}[theorem]{Proposition}
\newtheorem{corollary}[theorem]{Corollary}
\newtheorem{lemma}[theorem]{Lemma}

\theoremstyle{definition}
\newtheorem{definition}[theorem]{Definition}

\newtheorem{example}[theorem]{Example}

\theoremstyle{remark}
\newtheorem{remark}[theorem]{Remark}

\numberwithin{equation}{section}
\numberwithin{figure}{section}

\usepackage[all]{xy}
\CompileMatrices
\newdir{ >}{{}*!/-10pt/\dir{>}}

\def\mc{\mathcal}
\def\ms{\mathsf}
\def\Z{\mathbb Z}
\def\E{\mathbb E}
\def\ss{\subseteq}
\def\R{\mathbb R}

\def\es{\varnothing}
\def\Spec{\ms{Spec}}
\def\A{\mathbb A}
\def\lan{\langle}
\def\ran{\rangle}

\def\wt{\widetilde}

\def\wh{\widehat}
\def\udot{\sqcup}

\def\G{\mathbb G}
\def\ldot{.}
\def\Inv{\ms{Inv}}

\def\Oc{\mc O}

\def\res{\ms{res}}
\def\e{\epsilon}
\def\a{\alpha}
\def\La{\Lambda}
\def\b{\beta}
\def\ol{\overline}

\def\fmil{f^{\ms{Mil}}}

\def\lra{\leftrightarrow}

\def\Fk{\mc F_{k_0}}
\def\lan{\langle}
\def\ran{\rangle}

\renewcommand\le{\leqslant}
\renewcommand\ge{\geqslant}

\makeatletter
\newcommand{\addresseshere}{
	 \enddoc@text\let\enddoc@text\relax
	 }
	 \makeatother

	 \def\Vsg{V^{\star\Gamma}}

	 \CompileMatrices
	 \newdir{ >}{{}*!/-10pt/\dir{>}}

	 \newcommand{\smb}{{\scriptscriptstyle \bullet}}

	 \newcommand{\CM}{M}        
	 \newcommand{\CMu}[1]{M_{#1, \ms{unr}}}        
	 
	 \newcommand{\CK}{\mathrm{C}}         

	 \newcommand{\Fields}{\mc F }

	 \newcommand{\gk}{k_0}

\DeclareMathOperator{\Kt}{k} 
 \newcommand{\MK}{\Kt^{\ms M}}

\DeclareMathOperator{\Ker}{\mathrm{Ker}}


\begin{document}

\author{Christian Hirsch}
\address[Christian Hirsch]{University of Groningen, Bernoulli Institute, Nijenborgh~9, 9747 AG Groningen, The Netherlands}
\email{c.p.hirsch@rug.nl}

\thanks{This work is supported by The Danish Council for Independent Research | Natural Sciences, grant DFF -- 7014-00074 \emph{Statistics for point processes in space and beyond}, and by the \emph{Centre for Stochastic Geometry and Advanced Bioimaging}, funded by grant 8721 from the Villum Foundation.}

\title[Mod two Invariants of Weyl groups]{On the decomposability of mod 2 cohomological invariants of Weyl groups}

\begin{abstract}
We compute the invariants of Weyl groups in mod 2 Milnor $K$-theory and more general cycle modules, which are annihilated by 2. Over a base field of characteristic coprime to the group order, the invariants decompose as direct sums of the coefficient module. All basis elements are induced either by Stiefel-Whitney classes or specific invariants in the Witt ring. The proof is based on Serre's splitting principle that guarantees detection of invariants on elementary abelian 2-subgroups generated by reflections.
\end{abstract}

\maketitle

\goodbreak
\section{Introduction}
\label{introSec}

%
%

\noindent

Let $G$ be a smooth affine algebraic group over a field $k_0$ of characteristic not 2. Motivated from the concept of characteristic classes in
topology, the idea behind \emph{cohomological invariants} as presented by J.-P.~Serre in \cite{CohInv} is to provide tools for detecting that two
torsors are not isomorphic. Loosely speaking, such an invariant assigns a value in an abelian group to an algebraic object, such as a quadratic form or an \'etale algebra. 

\smallbreak

The formal definition of a cohomological invariant is due to J.-P.~Serre and appears in his lectures \cite{CohInv}, where also a brief account
of the history of the subject is given. First, we identify the pointed set of isomorphism classes of $G$-torsors over a field $k$ with the first
non-abelian Galois cohomology $H^1(k, G)$. Further, let $M$ be a functor from the category $\Fields_{\gk}$ of finitely generated field extensions
of $\gk$, to abelian groups. Then, a \emph{cohomological invariant} of $G$ with values in the coefficient space $M$ is a natural transformation
from $H^1( -, G)$ to $M( - )$ considered as functors on $\Fields_{\gk}$. Interesting examples of the functor $M$ include Witt groups or Milnor
$K$-theory modulo $2$, which is the same as Galois cohomology with $\Z/2$-coefficients by Voevodsky's proof of the Milnor conjecture.

\smallbreak

In general, the cohomological invariants of a given algebraic group with values in some functor $M$ are hard to
compute and there are only a few explicit computations carried out yet. One exception are the cohomological invariants
of the orthogonal group over a field of characteristic not $2$ with values in Milnor $K$-theory modulo $2$. These invariants are generated by Stiefel-Whitney classes
$$
w_i:\, H^1(-, O_n) \to K^{\ms M}_i( - )/2
$$
introduced by Delzant \cite{De62}. Now, every finite group $G$ embeds in a symmetric group $S_n$ for an appropriate $n$,
and this group in turn embeds in $O_n$. Pulling back the Stiefel-Whitney classes along such homomorphisms $G \to S_n \to O_n$
is a rich source of cohomological invariants of finite groups considered as group scheme of finite type over a base field $\gk$.

\smallbreak

In this work, we show that most cohomological invariants of a Weyl group $G$ over a field $\gk$ of characteristic coprime to $|G|$ arise in this way if
the coefficient space is a cycle module $\CM_\ast$ in the sense of Rost~\cite{Ro96}, which is annihilated by~$2$. More
precisely, there exists a finite family of invariants $\{a_i\}_{i \in I}$ with values
in $K_\ast^M/2$, such that every invariant $a$ over $\gk$ with values in $\CM_\ast$ decomposes uniquely as
$$
a = \sum\limits_{i \in I}a_i m_i, 
$$
for some constant invariants $m_i \in \CM_\ast(\gk)$.
In characteristic 0, any Weyl group is a product of the irreducible ones mentioned above. Hence, invoking a product formula of J.-P.~Serre
yields the decomposition for cohomological invariants.

\smallbreak

The proof of this result is constructive, in the sense that we give precise formulas for the generators $\{a_i\}_{i \in I}$. For most Weyl groups
the invariants are induced by Stiefel-Whitney classes coming from embeddings of the Weyl group into certain orthogonal groups. Note that
these embeddings make use of the fact that such a Weyl group can be realized as orthogonal reflection group over every field of characteristic not~$2$.
However, if the Weyl group has factors of type $D_{2n}$, $E_7$ or $E_8$, then besides Stiefel-Whitney classes also specific Witt-type invariants appear,
which induce invariants in mod 2 Milnor $K$-theory via the Milnor isomorphism. All basis
elements are invariants derived from either the Stiefel-Whitney or the Witt-ring invariants.

\smallbreak

Crucial for the derivation is Serre's splitting principle for Weyl groups: if two invariants coincide on the elementary abelian $2$-subgroups generated by reflections, then these are the same. This allows the following proof strategy. Since Stiefel-Whitney classes and Witt invariants provide us with a family of invariants, we only have to show that a given invariant coincides on the elementary abelian subgroups with a combination
from this list. The invariants are then computed case by case for the various types.

\smallbreak

J.-P.~Serre has recently computed  with a different method the invariants of Weyl groups with values in Galois cohomology, see his 2018 Oberwolfach talk \cite{Se18}.
 In an e-mail exchange on an earlier version of the present paper, J.-P.~Serre explains how to remove many of the restrictions on the characteristic of $k_0$. An excerpt of his letter is reproduced in Section \ref{serreSec}.  J.~Ducoat provided a proof of Serre's splitting principle and attempted to compute the invariants for groups of type $B_n$ and $D_n$ \cite{Du11}. However, many proofs are incomplete as they are ``left to the reader'' or ``similar to previous ones''. Moreover, Theorem 5 on page 4 about the invariants of $W(D_n)$ is not correct as stated, because an invariant in degree $n/2$ is missing. Therefore, we provide detailed computations also for the types $B_n$ and $D_n$.

\medbreak

The content of this article is as follows. In Section \ref{resultsSect}, we state the main result and fix notations and conventions. Next, Section \ref{techLemSec} contains preliminary results. The proof of the main result occupies the rest of the paper. It also includes an appendix, elucidating how to use a {\tt GAP}-program to determine the invariants for $E_7$ and $E_8$.

\smallbreak

\bigbreak

\section*{Acknowledgments}
The present manuscript has a long history. It is a condensed version of my diploma thesis at LMU Munich supervised by F.~Morel. I am very grateful for his comments and insights that shaped this work in many ways. The thesis is available online and contains additional background material from algebraic geometry \cite{Hi10} as well as results for reflection groups that are not of Weyl type. Moreover, I thank S.~Gille for massive help and discussions on earlier versions of the manuscript. He was also the one to mention the thesis during a presentation of J.-P.~Serre at the 2018 Oberwolfach meeting. I am very grateful to J.-P.~Serre for a highly insightful e-mail exchange and for sharing with me an early version of his report \cite{Se18}. His remarks helped to both substantially raise the quality of the presentation, and also improve the contents such as removing restrictions on the characteristic in the present paper. Moreover, an earlier version also contained an irritating assumption that $-1$ be a square in $k_0$. Thanks to a more appropriate representation of $W(B_2)$ pointed out by J.-P.~Serre, also this assumption could be removed in the present version. Finally, I thank the anonymous referee for the careful reading of the manuscript and valuable observations that helped to improve the presentation.

\bigbreak

\bigbreak
\newpage
\begin{center}
{\bf\large
Part I: Results and methods.
}
\end{center}

\section{Main theorem and proof strategy}
\label{resultsSect}

\subsection{Cycle modules}
We consider in this work invariants with values in a cycle module $\CM_\ast$ in the sense of Rost, which is annihilated by $2$. Recall that a cycle module over a field $\gk$ is a covariant functor
$$
k\, \longmapsto\, \CM_\ast(k)\,: = \;\bigoplus\limits_{n \in \Z}\CM_n(k)
$$
on the category $\Fields_{\gk}$ with values in graded Milnor $K$-theory modules. For a field extension $\iota:k\ss L$, the image of $z \in \CM_\ast(k)$ in $\CM_\ast(L)$ is denoted by $\iota_\ast(z)$. By definition, cycle modules have further structure and we refer the reader to \cite{Ro96} for details.

The main example of a cycle module is Milnor $K$-theory: 
\begin{align*}
\Fk & \to \text{$\Z$-graded rings}\\
k &\mapsto K^{\ms M}_*(k) = \oplus_{n \ge 0} K^{\ms M}_n(k).
\end{align*}
For $a_1, \dots, a_n \in k^\times$, we denote pure symbols in $K^{\ms M}_n(k)$ by $\{a_1, \dots, a_n\}$.
The graded abelian group $\CM_\ast(k)$ has the structure of a graded $K^{\ms M}_\ast(k)$-module for every
field $k \in \Fields_{\gk}$. Hence, if $\CM_\ast$ is annihilated by 2, it becomes a $K^{\ms M}_\ast(k)/2$-module. For ease of notation, we set $\MK_\ast(k): = K^{\ms M}_\ast(k)/2$ and denote the image of a symbol $\{ a_1, \dots, a_n\} \in K^{\ms M}_n(k)$ in $\MK_n(k)$ by $\{a_1, \dots, a_n\}$. We say that
$\CM_\ast$ has a \emph{$\MK_\ast$-structure} if $\CM_\ast$ is annihilated by 2.

\bigbreak

\noindent
{\it
From now on cycle module means cycle module with $\MK_\ast$-structure.
}

\bigbreak

\subsection{Invariants with values in cycle modules}
Let $G$ and $\CM_\ast$ be a linear algebraic group and a cycle module over $\gk$, respectively.
Recall from Section \ref{introSec} that a \emph{cohomological invariant} of $G$ with values in $\CM_n$ is
a natural transformation from $H^1(\, -, G)$ to $\CM_n(\, -\, )$. We denote the set of all invariants of degree
$n$ of $G$ with values in $\CM_\ast$ by $\Inv^n(G, \CM_\ast)$, and set 
$$\Inv(G, \CM_\ast):= \Inv_{\gk}(G, \CM_\ast):= \bigoplus\limits_{n \in \Z}\Inv^n(G, \CM_\ast).$$ 

\smallbreak

For $k \in \Fields_{\gk}$, any invariant $a \in \Inv_{\gk}(G, \CM_\ast)$ restricts to a natural transformation of functors
$H^1(\, -, G) \to \CM_\ast(\, -\, )$ on the full sub-category $\Fields_k$ of $\Fields_{\gk}$. We denote this restricted invariant
by $\res_{k/\gk}(a)$ or by the same symbol $a$ if the meaning is clear from the context.
A particular example of invariants are the \emph{constant invariants}, which are in one-to-one correspondence with elements of $\CM_\ast(\gk)$:
The constant invariant $c \in \CM_\ast(\gk)$ maps every $x \in H^1(k, G)$ onto the image of $c$ in $\CM_\ast(k)$ for all $k \in \Fields_{\gk}$. 
The set $\Inv(G, \CM_\ast)$ is a $\MK_\ast(\gk)$-module, so that  if $a:\, H^1(\, -, G) \to \MK_\ast(\, -\, )$ is a Milnor $K$-theory invariant of degree $m$ and $x \in \CM_n(\gk)$, then
$$
a \cdot x\,:\; H^1(k, G)\, \to \, \CM_{m + n}(k), \; T\, \mapsto a_k(T) x_k
$$
is an invariant with values in $\CM_\ast$ of degree $m + n$.
We now define precisely what it means that an invariant can be represented uniquely as a sum of basis elements.

\begin{definition}
\label{decomposableDef}
Let $\CM_\ast$ be a cycle module over the field $\gk$, and $G$ a linear algebraic group over $\gk$.

\smallbreak

\begin{itemize}
	\item[(i)]
		A subgroup $S\ss\Inv_{\gk}^{\ast}(G,\CM_{\ast})$ is a \emph{free $\CM_{\ast}(\gk)$-module} 		with \emph{basis} $a^{(i)} \in \Inv^{d_i}_{\gk}(G, \MK_{\ast})$, $i\in I$, if
		\begin{align*}
			\bigoplus_{i\in I}\CM_{*- d_i}(\gk) \to S, \qquad\qquad
			 \{m_i\}_{i \in I} \mapsto \sum_{i \le r} a^{(i)}\cdot m_i
		\end{align*}
		is an isomorphism of abelian groups.
\smallbreak

\item[(ii)]
$\Inv(G, \CM_\ast)$ is \emph{ completely decomposable} with a finite {basis} $a_i \in \Inv^{d_i}_{\gk}(G, \MK_\ast)$ if $\Inv_k^\ast(G, \CM_\ast)$ is a free $\CM_\ast(k)$-module with the corresponding basis $\res_{k/\gk}(a_i) \in \Inv_k^{d_i}(G, \MK_\ast)$, $i \in I$, for all $k \in \Fields_{\gk}$.
\end{itemize}
\end{definition}

\smallbreak

After these preparations, we now state the main result.

\begin{theorem}
\label{reflThm}
Let $G$ be an irreducible Weyl group. Let $\gk$ be a field of characteristic coprime to $|G|$ and $\CM_\ast$ a cycle
module over $\gk$. Then, $\Inv^\ast_{\gk}(G, \CM_\ast)$ is completely decomposable.
\end{theorem}

The proof of Theorem \ref{reflThm} is constructive and we describe the generators explicitly. These depend on the type of the Weyl group and will be given in the course of the computation later on. Now, we explain the strategy starting with a reminder on Weyl groups. 

Let $\E$ be a finite-dimensional real vector space with scalar product $( -, - )$ and orthogonal group $O(\E)$. Then, $s_v:\E \to\E$,
\begin{align*}
	s_v(w):= w - \frac{2(v, w)}{(v, v)} v, 
\end{align*}
defines the reflection at a vector $v \in \E$ with $(v, v) \ne 0$.

\smallbreak

Now, the \emph{Weyl group} $W(\Sigma)$ associated with a crystallographic root system $\Sigma\ss\E$ is the subgroup of $O(\E)$ generated by all reflections $s_\a$ at the roots $\a \in \Sigma$. By definition of a root system, the scalars ${2(\a, \b)}/{(\a, \a)}$ are integers for all $\a, \b \in \Sigma$ and the reflections act on the root system. The Weyl group is {\it irreducible} if the corresponding root system is irreducible. 

\smallbreak

The irreducible root systems are classified by types $A_n, B_n, C_n, D_n, E_6, E_7$, $E_8$, $F_4, G_2$. Let $\Sigma$ be such an irreducible root system. Then, there exists an Euclidean space $\E=\R^n$ for an appropriate $n$, such that (i) $\Sigma\ss V:=\bigoplus\limits_{i \le n}\Z [1/2]e_i$, where $e_1,\ldots ,e_n$ is the standard basis of $\R^n$, and (ii) $W(\Sigma)$ maps $V$ into itself. This can be deduced using the realizations of these root systems in Bourbaki \cite[PLATES I-VIII]{LIE4-6}. If now $\gk$ is a field of characteristic not $2$ then $W(\Sigma)$ acts via scalar extension on $V_{\gk}:=\gk\otimes_{\Z [1/2]}V$ and can so be realized as orthogonal reflection group over $\gk$ considering $V_{\gk}$ has regular bilinear space with the scalar product induced by the restriction of the standard scalar product of $\E = \R^n$ to $V$.

\medbreak

The strategy of proof for an irreducible Weyl group $G$, is as follows. We leverage different embeddings of the Weyl group $G$ into an orthogonal group $O_n$ over the field $\gk$. Now, the invariants of $O_n$ with values in $\MK_\ast$ are generated by the Stiefel-Whitney classes, see \cite{CohInv}. Considering embeddings $W\hookrightarrow O_n$ gives rise to a family of invariants in $\Inv(G, \MK_\ast)$ by composing the Stiefel-Whitney classes with the natural transformation $H^1( -, W) \to H^1( -, O_n)$. As we shall see in Sections \ref{BnSubSect} -- \ref{E6-8SubSect}, these already generate $\Inv(G, \CM_\ast)$ except if $G$ is of type $D_{2n}$, $E_7$, or $E_8$. The 'missing' invariants have their source in certain Witt invariants. 

\medbreak

Having a family of invariants with values in $\MK_\ast$ at our disposal, we deduce Theorem \ref{reflThm} for an irreducible Weyl group $G$ by showing that this set of invariants contains a basis of $\Inv(G, \CM_\ast)$ in the sense of Definition \ref{decomposableDef}. The main tool is the following adaptation of Serre's splitting principle, which is proven in \cite[Corollary 4.10]{GiHi19}.	Loosely speaking, if $\gk$ is a field of characteristic coprime to $|G|$, then $\Inv(G, \CM_\ast)$ is detected by the maximal elementary abelian $2$-subgroups of $G$ generated by reflections.  We let $\Omega(G)$ denote the set of conjugacy classes of maximal elementary $2$-abelian subgroups of $G$, which are generated by reflections.

\smallbreak

Note that the proof of Theorem \ref{reflThm} for Weyl groups of type $G_2$ in Section \ref{G2InvSubSect} is purely group theoretic, in the sense that it uses only its semi-direct decomposition and not the geometry of the corresponding root system.

\smallbreak
%
%
\begin{proposition}[Serre's splitting principle]
\label{splitCorollary}
Let $\CM_\ast$ be a cycle module over $\gk$ and $G$ be a Weyl group. Let $\gk$ be a field of characteristic coprime to $|G|$.
 Then, the canonical map
\begin{align*}
	\big({\ms{res}^P_G}\big)_{[P]}:\, \Inv(G, \CM_\ast)& \to \prod_{[P] \in \Omega(G)} \Inv(P, \CM_\ast)^{N_G(P)}
\end{align*}
is injective, where $N_G(P)$ is the normalizer of the maximal elementary $2$-abelian subgroup $P$ of $G$, which is generated by reflections.
\end{proposition}

\smallbreak

\noindent
We point out that the assumption that order of the irreducible Weyl group $G$ and the characteristic of $\gk$ are coprime seems to be not necessary, see Section~\ref{serreSec}. This assumption comes from the article \cite{GiHi19}, where the splitting principle is proven for more general orthogonal reflection groups. This would also remove that assumption from Theorem \ref{reflThm}.

\medbreak

\begin{remark}
\label{splittingPrincipleRem}
For groups of type $A_n$, $D_n$, $E_6$, $E_7$, or $E_8$, any two roots are conjugate \cite[Rem.\ 4, Sect.\ 2.9]{Hu90}.
Hence, an induction argument shows that for these types, there is up to conjugacy only one maximal abelian $2$-subgroup $P$ generated by reflections.
In particular, by Proposition \ref{splitCorollary}, the restriction map $\ms{res}_G^P$ is injective for simply-laced groups.
\end{remark}

\medbreak

The computation of the invariants of an arbitrary Weyl group follows from Theorem~\ref{reflThm} by a product formula of Serre.
To state the product formula precisely, we first introduce the notion of a product of invariants. Identifying $H^1(k, G'\times G)$
with $ H^1(k, G') \times H^1(k, G)$, for invariants $a \in \Inv_{\gk}(G, \MK_\ast)$ and $b \in \Inv_{\gk}(G', \CM_\ast)$, we define the product $ab$ through 
\begin{align*}
	(ab)_k:\,H^1(k, G\times G') &\to \CM_\ast(k)\\ 
	(T, T')&\mapsto a_k(T) b_k(T').
\end{align*}
\begin{proposition}[Product formula]
\label{productLem}
Let $\CM_\ast$ be a cycle module and $G, G'$ algebraic groups over $\gk$. If $\Inv^\ast_{\gk}(G, \CM_\ast)$ is completely decomposable with finite basis $\{a_i\}_{i \in I}$, then the map 
	\begin{align*} 
\bigoplus_{i \in I}\Inv^\ast_k(G', \CM_\ast) & \to \Inv^\ast_k(G \times G', \CM_\ast)\\
\{b_i\}_{i \in I } &\mapsto \sum_{i \in I }\res_{k/\gk}(a_i) b_i
\end{align*}
is an isomorphism for all $k \in \Fields_{\gk}$. In particular, if the invariants of both $G$ and $G'$ are completely decomposable, then so is $\Inv^\ast_{\gk}(G \times G', \CM_\ast)$.
\end{proposition}
\begin{proof}
	We follow the outline given in \cite[Part I, Exercise 16.5]{CohInv}. Replacing $a_i$ by $\res_{k/\gk}(a_i)$ we can assume $k = \gk$.

\smallbreak

To show surjectivity, let $a \in \Inv^\ast_{\gk}(G \times G', \CM_\ast)$. Then, for every $k \in \Fk$ and $T' \in H^1(k, G')$ we define an invariant $\bar a \in \Inv^\ast_k(G, \CM_\ast)$ by mapping $T \in H^1(\ell, G)$ to $\bar a_\ell(T) = a_\ell(T \times T'_\ell)$, where, $T'_\ell$ denotes the image of $T'$ in $H^1(\ell, G')$ under the base change map. Since $\Inv(G, \CM_\ast)$ is completely decomposable, $\bar a$ can be uniquely expressed as $\sum_i \res_{k/\gk}(a_i) b_i(T')$ for suitable $b_i(T') \in \CM_\ast(k)$. It remains to prove that $b_i \in \Inv(G', \CM_\ast)$ for all $i$. To achieve this goal, let $\iota:k\ss k_1$ be a field extension in $\Fields_{\gk}$ and $T' \in H^1(k, G')$. Then,
$$
\iota_\ast\Big(\sum\limits_{i \in I}\res_{k/\gk}(a_i)(T) b_i(T')\Big) = 
\sum\limits_{i \in I}\res_{k_1/\gk}(a_i)(T_{k_1}) b_i(T'_{k_1}).
$$
 Since $a_i$'s are invariants
$$
\sum\limits_{i \in I}\res_{k_1/\gk}(a_i) \iota_\ast(b_i(T'))\, = \, \sum\limits_{i \in I}\res_{k_1/\gk}(a_i) b_i(T'_{k_1}).
$$
As the $a_i$'s are a basis we get $b_i(T'_{k_1}) = \iota_\ast(b_i(T'))$, as asserted.

\smallbreak

To show injectivity, we assume $\sum\limits_{i \in I}a_ib_i = 0$ and claim that $b_i = 0$ for all $i \in I$. Fix a field $k$ and $T' \in H^1(k, G')$.
Then $\sum\limits_{i \in I}a_ib_i(T') \in \Inv^\ast_k(G, \CM_\ast)$ is the constant zero invariant. Since the $a_i$'s are a basis, we get $b_i(T') = 0$
for all $i \in I$. Since $k$ and $T'$ were arbitrary, this implies that the $b_i$'s are constant zero. 
\end{proof}

\smallbreak

Since every Weyl group is a product of irreducible ones, we get the following corollary.
\begin{corollary}
Let $\gk$ be a field of characteristic coprime to $|G|$ and $\CM_\ast$ a cycle module over $\gk$. Then, $\Inv^\ast_{\gk}(G, \CM_\ast)$ is completely decomposable for all Weyl groups $G$.
\end{corollary}

\goodbreak
\section{Preparations for the proof}
\label{techLemSec}

In this section, we establish several key lemmas on cycle modules. We also discuss auxiliary results used in the type-by-type proof of Theorem \ref{reflThm} for irreducible Weyl groups.

\smallbreak

\subsection{Cycle complex computations.}
We start with a computation of cycle module cohomology which seems to be well known, but for which we have not found an appropriate reference. To this end, we recall first the cycle complex associated with a cycle module $M_\ast$ over $\gk$. We refer the reader to Rost \cite{Ro96} for further details. 

\smallbreak

Let $X$ be a scheme essentially of finite type over $\gk$. That is, $X$ is of finite type over $\gk$ or the localization of such a $\gk$-scheme. Then, the \emph{cycle complex} is given by
$$
\bigoplus\limits_{x \in X^{(0)}}\CM_n(\gk (x)) \xrightarrow{d^0_{X, n}} \bigoplus\limits_{x \in X^{(1)}}\CM_{n - 1}(\gk (x)) 
\xrightarrow{d^1_{X, n}} \bigoplus\limits_{x \in X^{(2)}}\CM_{n - 2}(\gk (x))\to \cdots, 
$$
where $X^{(p)}\ss X$ denotes the set of points of codimension $p \ge 0$ in $X$ and $\gk(x)$ is the residue field of $x \in X$. In general, the differentials $d^p_{X, n}$ are sums of composition of second residue maps and transfer maps. If $X$ is an integral scheme with function field $\gk (X)$ and regular in codimension 1, then the components of $d^0_{X, n}$ are the \emph{second residue maps} $\partial_x:\,\CM_n(\gk (X)) \to \CM_{n - 1}(\gk (x))$. In particular, the cohomology group in dimension 0, also called \emph{unramified cohomology} of $X$ with values in $\CM_n$, equals
$$
\CMu n(X):= \Ker\Big( \CM_n(\gk(X))\xrightarrow{\:(\partial_{x})_{x \in X^{(1)}}\;}
\bigoplus\limits_{x \in X^{(1)}}\CM_{n - 1}(\gk (x))\, \Big).
$$
 In case $X = \Spec(R)$, we use affine notations and write $\CM_{n, \ms{unr}}(R)$ instead of $\CMu n(X)$.

\smallbreak

\begin{lemma}
\label{cheeseLine}
Let $\CM_\ast$ be a cycle module over $k_0$ and $R$ a regular and integral $k_0$-algebra with fraction field $K$, which is essentially of finite type. Let $a_1, \dots, a_l \in R$ be such that $a_i - a_j \in R^\times$ for all $i \ne j$. Then, 
$$
	\CMu n(R[T]_{\prod\limits_{i \le l }(T - a_i)}) \simeq\, \CMu n(R)\;\oplus\; 	\bigoplus\limits_{i \le l}\{ T - a_i\} \cdot\CMu {n - 1}(R), 
$$
	where we consider $\{ T - a_i\}$ as an element of $K_1^{\ms M}(K(T))$ and $\CMu{n - 1}(R)$ as a subset of $\CM_{n - 1}(K(T))$.
\end{lemma}
\begin{proof}
Setting $f(T): = \prod\limits_{i \le l }(T - a_i)$, we consider the following short exact sequence of cycle complexes, where for a cohomological complex $P^{\smb}$ we denote by $P^{\smb}[1]$ the shifted complex with $P^i$ in degree $i + 1$:
$$
\xymatrix{
\CK^{\smb}(R[T]/R[T] \cdot f(T), \CM_{n - 1})[1]\;\; \ar@{>->}[r] & \CK^{\smb}(R[T], \CM_n) \ar@{->>}[r] &
 \CK^{\smb}(R[T]_{f(T)}, \CM_n).
}
$$
Using homotopy invariance, the associated long exact cohomology sequence starts with
$$
	0 \to \CMu n(R) \to \CMu n(R[T]_{f(T)}) \to \CMu {n - 1}(R[T]/R[T] \cdot f(T)).
$$
We claim that the map on the right-hand side of this exact sequence is a split surjection. Indeed, by the Chinese remainder theorem, $$
	R[T]/R[T] \cdot f(T)\, \simeq\, \prod\limits_{i \le l} R[T]/R[T] \cdot (T - a_i)\, \simeq\, \prod\limits_{i \le l}R, 
$$
	so that $\CMu{n - 1}(R[T]/R[T] \cdot f(T))\simeq\CMu{n - 1}(R)^{\oplus\, l}$. Disentangling the definitions of the appearing maps shows that 
$$
		\CMu {n - 1}(R)^{\oplus l} \to \CMu n (R[T]_{f(T)}),\quad 
		(x_1, \dots, x_l) \longmapsto \sum\limits_{i \le l}\{T - a_i\} x_i
	$$
	defines the asserted splitting.
	\end{proof}

	\medbreak

	By induction and homotopy invariance, Lemma \ref{cheeseLine} implies the well-known computation of the unramified cohomology of a Laurent ring.

	\bigbreak

	\noindent
	\begin{corollary}
	\label{Z2Invariants}
	Let $\CM_\ast$ be a cycle module over $k_0$. Then, 
	$$
		\CMu{n}(k_0 [T_1^{\pm}, \dots, T_l^{\pm}]) \simeq\		\bigoplus_{\substack{r \le l \\ 1 \le i_1 < \cdots < i_r \le l}}
	\{ T_{i_1}, \dots, T_{i_r}\} \cdot\CM_{n - r}(k_0).
	$$
	\end{corollary}

	\subsection{Invariants of $(\Z/2)^n$}
	\label{Z/2InvSec}
	Corollary \ref{Z2Invariants} implies that the invariants of $(\Z/2)^n$ with values in a cycle module are completely decomposable. This is shown for invariants of $(\Z/2)^n$ with values in $\MK_\ast$ in Serre's lectures \cite[Part I, Sect.\ 16]{CohInv}. Writing $(\a) \in H^1(k, \Z/2)$ for the class of $\a \in k^\times$, every index set $1 \le i_1 < \cdots < i_l \le n$ gives rise to an invariant 
\begin{align*}
x_{i_1, \dots, i_l}:\, H^1(k, (\Z/2)^n) \simeq H^1(k, \Z/2)^n \to \MK_l(k)\\
\big[ (\a_1), \dots, (\a_n)\big] \mapsto \{\a_{i_1}, \dots, \a_{i_l}\} .
\end{align*}
We show that they form a basis of $\Inv((\Z/2)^n, \CM_\ast)$ for every cycle module $\CM_\ast$ with $\MK_\ast$-structure.

\smallbreak

Let $k \in \Fields_{\gk}$, $a \in \Inv^\ast_k((\Z/2)^n, \CM_\ast)$ and write $K: = k(t_1, \dots, t_n)$ for the rational function field in $n$ variables over the field $k$. Then, $T: k(\sqrt{t_1}, \dots, \sqrt{t_n})\supseteq k(t_1, \dots, t_n)$ is a versal $(\Z/2)^n$-torsor, so that by \cite[Part I, Thm.~11.1]{CohInv} or \cite[Thm.\ 3.5]{GiHi19},
$$
a_K(T)\, \in \, \CMu{\ast}(k[t_1^{\pm}, \dots, t_n^{\pm}]).
$$
By Corollary \ref{Z2Invariants}, there exist unique $m_{i_1,\dots, i_l} \in \CM_\ast(k)$ with
$$
a_K(T) = \sum_{\substack{l \le n \\ 1 \le i_1 < \dots < i_l \le n}}\big\{ t_{i_1}, \dots, t_{i_l}\big\} m_{i_1,\dots, i_l}.
$$
Then, the invariant 
$$
b: = \sum_{\substack{l \le n \\ 1 \le i_1 < \dots < i_l \le n}}x_{i_1,\dots, i_l} m_{i_1,\dots, i_l}.
$$
agrees with $a$ on the versal torsor $T$. Hence, the detection principle in the form of \cite[Part I, 12.2]{CohInv} or \cite[Thm.\ 3.7]{GiHi19} implies that $a = b$, as asserted.

\subsection{Invariants of Weyl groups of type $G_2$}
\label{G2InvSubSect}
Assume here that the base field is of characteristic not~$2$ or~$3$.

The group $W(G_2)$ is a semi-direct product of a normal subgroup $L$ of order $3$ and a subgroup $P \simeq (\Z/2)^2$
generated by the reflections at two orthogonal roots, see \cite[Chap.\ VI, \S 4, No 13]{LIE4-6}. Since there is up to conjugacy only one such $P$,
Proposition \ref{splitCorollary} shows that the restriction map $\res_{W(G_2)}^P$ is injective. Since the projection $W(G_2)\simeq P\ltimes L \to P$
induces a splitting, we deduce that $\res_{W(G_2)}^P$ is in fact an isomorphism.

In view of the results for other Weyl groups it is worthwhile to note that a basis for the invariants can also be expressed in terms of the Stiefel-Whitney invariants to be introduced in Section \ref{swInvSubSect} below.  As in Section \ref{b2Sec} below, we see that the restriction of the Stiefel-Whitney classes in degrees 1 and 2 to $P$ correspond to the invariants $x_1 + x_2$ and $x_{1, 2}$. Finally, considering the morphism $W(G_2) \to O_1 = \{\pm 1\}$ sending one of the two classes of reflections to $-1$ and the other to 1 yields the invariant $x_1$ (or $x_2$).

%
%
\subsection{Torsor computations}
Henceforth, we switch freely between the interpretation of $H^1(k, O_n)$ via cocycles on the one hand and via quadratic forms on the other hand. For this purpose, we recall how to view $H^1(k, O_n)$ in terms of non-abelian Galois cohomology \cite{cg}. Let $c \in Z^1(\Gamma, O_n)$ be a cocycle. That is, $c$ is a continuous map from the absolute Galois group $\Gamma$ of a separable closure $k_s/k$ to $O_n(k_s)$ and satisfies the cocycle condition $c_{\sigma\tau} = c_\sigma \cdot\sigma(c_\tau)$. To construct a quadratic form $q_c$ over $k$, we first define an action $\star$ of $\Gamma$ on $k_s^n$ via $\sigma \star v = c_\sigma(\sigma(v))$. Then, we let $v_1, \dots, v_n \in k_s^n$ denote a $k$ basis of the vector space 
\begin{align}
	\label{vStarEq}
	\Vsg = \{v \in k_s^n:\, \sigma \star v = v\text{ for all }\sigma \in \Gamma\}.
\end{align}
Now, we let $q_c$ be the quadratic form whose associated bilinear form $b_{q_c}$ is determined by $b_{q_c}(e_i, e_j) = \lan v_i, v_j \ran$, where $\lan \cdot, \cdot\ran$ denotes the standard scalar product in $k_s^n$.{ In other words, $q_c$ is the restriction to $\Vsg$ of the quadratic form associated with the standard scalar product $\lan \cdot, \cdot \ran$.} We will come back frequently to the following three pivotal examples, where $V = k^2_s$.

\medbreak

\begin{example}
\label{abQuadratic}
Consider the group homomorphism $(\Z/2)^2 \to O_2$,
\begin{align*}
e_1\mapsto \begin{pmatrix} 0&1\\1&0\end{pmatrix}&, \; e_2\mapsto \begin{pmatrix} 0&-1\\-1&0\end{pmatrix}.
\end{align*}
	Let $(\a, \b) \in (k^\times/{k^\times}^2)^2$ be a $(\Z/2)^2$-torsor over $k$. Then, $v_1 = (\sqrt\a, -\sqrt\a)^\top$, $v_2 = (\sqrt\b, \sqrt\b)^\top$ defines a basis of $V^{\star\Gamma}$ and the induced bilinear form is the diagonal form $q_{(\a, \b)} = \lan 2\a, 2\b\ran$.
\end{example}

\begin{example}
\label{aQuadratic}
Consider the group homomorphism $\Z/2 \to O_2$,
\begin{align*}
e_1&\mapsto \begin{pmatrix} 0&1\\1&0\end{pmatrix}.\\
\end{align*}
Let $\a \in k^\times/k^{\times2}$ be a $\Z/2$-torsor. Applying the above example with $\b = 1$, we see that the induced bilinear form is the diagonal form $q_{(\a)} = \lan 2\a, 2\ran$.
\end{example}

\begin{example}
\label{abQuadratic2}
Consider the group homomorphism $(\Z/2)^2 \to O_2$,
\begin{align*}
e_1\mapsto \begin{pmatrix} 0&1\\1&0\end{pmatrix}&, \;e_2\mapsto \begin{pmatrix} 0&1\\1&0\end{pmatrix}.
\end{align*}
Let $(\a, \b) \in (k^\times/{k^\times}^2)^2$ be a $(\Z/2)^2$-torsor over $k$. Then, $v_1 = (1, 1)^\top$, $v_2 = (\sqrt{\a\b}, -\sqrt{\a\b})^\top$ defines	a basis of $V^{\star\Gamma}$. The induced bilinear form is the diagonal form $q_{(\a, \b)} = \lan 2, 2\a\b\ran$.
\end{example}

%
%
\subsection{An embedding of $S_{2^n}$ into $O_{2^n}$}
\label{Sym - OnEmbeddingSubSect}
Next, we describe a specific embedding $(\Z/2)^n \to O_{2^n}$ on the torsor level. For any $ l \le 2^n - 1$ let $b(l) \ss [0, n - 1]$ be the position of the bits in the binary representation. That is, $l = \sum_{i \in b(l)} 2^i$. Furthermore, let $f_S$ be the flipping the bits at all positions in $S \ss [0, n - 1]$. In other words, $f_S:\, [0, 2^n - 1] \to [0, 2^n - 1]$,
\begin{align*}
	f_S(l):= b^{-1}(b(l) \Delta S), 
\end{align*}
where $R \Delta S = (R \setminus S) \cup (S \setminus R)$ is the symmetric difference. In this notation, the group homomorphism $\phi:\, (\Z/2)^n \to S_{2^n} \ss O_{2^n}$
\begin{align*}
	\phi\Big(\sum_{s \in S}e_s\Big):=f_S
\end{align*}
induces a map $\phi_*:\, H^1(k, (\Z/2)^n) \to H^1(k, O_{2^n})$, which we now describe explicitly.
%
%
\begin{lemma}
\label{pfisterLemma}
Let $\e_0, \dots, \e_{n - 1} \in k^\times / k^{\times2}$. Then, 
	$$\phi_*(\e_0, \dots, \e_{n - 1}) = \lan 2^n \ran \otimes \lan\lan -\e_0\ran\ran \otimes \lan\lan -\e_1\ran\ran \otimes \cdots \otimes \lan\lan -\e_{n - 1}\ran\ran.$$
\end{lemma}

\noindent
Since any two simply transitive actions on $[0, 2^n - 1]$ are conjugate in $S_{2^n}$, Lemma \ref{pfisterLemma} is more useful than it may seem at first.

\begin{proof}
	Consider a cocycle representation $c \in Z^1(\Gamma, (\Z/2)^n)$ of the torsor $(\e_0, \dots, \e_{n - 1}) \in (k^\times/k^{\times2})^n$. That is, the $i$th component
	of $c_\sigma$ equals 1 if and only if $\sigma\big(\sqrt{\e_i}\big) = - \sqrt{\e_i}$. To determine the quadratic form defined by the induced cocycle
	$\sigma \mapsto \phi(c_\sigma)$, we assert that a basis of the $k$-vector space $V^{\star \Gamma}$ from \eqref{vStarEq} is given by $\{v_0, \dots, v_{2^n - 1}\}$, 
	where $v_p$ has components
	$$(v_p)_\ell = (-1)^{|b(p) \cap b(\ell)|} \prod_{\substack{i \in b(p)}} \sqrt{\e_i}.$$
	First, $v_p \in V^{\star\Gamma}$, since writing $c_\sigma = \sum_{i \in S}e_i$ for some $S = S(\sigma) \ss [0, n - 1]$ shows that
\begin{align*}
 \sigma\Big((-1)^{|b(p) \cap b(\ell)|}\prod_{\substack{i \in b(p)}}\sqrt{\e_i}\Big) = (-1)^{|b(p) \cap b(\ell)| + |b(p) \cap S|}\prod_{i \in b(p)}\sqrt{\e_i} = (v_p)_{f_S(\ell)}.
\end{align*}
Moreover, to prove the linear independence of the $\{v_p\}_p$, we note that
\begin{align*}
b(v_p, v_p) = \sum_{u \le 2^n - 1} (v_p)_u(v_p)_u = 2^n\prod_{i \in b(p)}\e_i.
\end{align*}
Hence, it suffices to show that $b(v_p, v_q) = 0$, if $p \ne q$. By assumption, there is at least one $i \in b(p) \Delta b(q)$, so that pairing any $L \ss [0, n - 1]\setminus \{i\}$ with $L \cup \{i\}$ shows that
\begin{align*}
	b(v_p, v_q)& = \prod_{i \in b(p)}\sqrt{\epsilon_i} \cdot\prod_{i \in b(q)}\sqrt{\epsilon_i} \cdot\, \sum_{L \ss [0, n - 1]}(-1)^{|b(p) \cap L| + |b(q) \cap L|}\\
	& = \prod\limits_{\substack{i \in b(p) \\ j \in b(q)}}\sqrt{\epsilon_i\epsilon_j}\hspace{-.4cm} \sum_{L \ss [0, n - 1] \setminus \{i\}}\hspace{-.4cm}\big((-1)^{|b(p) \cap L| + |b(q) \cap L|} + (-1)^{|b(p) \cap L| + |b(q) \cap L| + 1}\big), 
\end{align*}
vanishes as claimed.
\end{proof}

\subsection{Stiefel-Whitney Invariants}
\label{swInvSubSect}

The \emph{total Stiefel-Whitney class} is defined by
\begin{align*}
w_\ast\,:\;
H^1(k, O_n)& \to \MK_\ast(k)\\
\lan \a_1, \dots, \a_n\ran &\mapsto \prod_{i \le n} (1 + \{\a_i\}), 
\end{align*}
where $\lan\a_1, \dots, \a_n\ran$ is the class in $H^1(k, O_n)$ of the diagonal form. They generate the invariants of the orthogonal group $O_n$ with values in $\MK_\ast$ as Serre shows in \cite[Part I, Sect.\ 17]{CohInv}.

\smallbreak

\begin{theorem}
\label{swThm}
Let $\gk$ be a field of characteristic not $2$. Then, the Stiefel-Whitney invariants form a basis in the sense of Definition~\ref{decomposableDef} of $\Inv(O_n, \MK_\ast)$ for all $n \ge 1$.
\end{theorem}

\smallbreak

By \cite[Rem.\ 17.4]{CohInv} the product of Stiefel-Whitney classes is given by 
\begin{equation}
\label{SWProduct}
	w_r w_s = \{-1\}^{b^{-1}(b(r) \cap b(s))} w_{r + s - b^{-1}(b(r) \cap b(s))},
\end{equation}
where $b( \cdot)$ denote the binary representation of Section \ref{Sym - OnEmbeddingSubSect}.

\begin{example}
\label{modSW}
Later, we will meet some examples where it is easier to do the computations with a slight variant of the Stiefel-Whitney classes.
Therefore, we introduce \emph{modified} Stiefel-Whitney classes $\wt{w_d} \in \Inv^d(O_n, \MK_\ast)$: For even $n$, we put $\wt{w_d}(q):= w_d(\lan 2\ran\otimes q)$ for all $d \le n$ and for odd $n$, we set inductively $\wt{w_0} = 1$ and $\wt{w_{d + 1}}(q) = w_{d + 1}(\lan2\ran\otimes q)-\{2\}\wt{w_d}(q)$. Then, we obtain for even $\ms{rank}(q)$ that
	$$ \wt{w_d}(\lan 2 \ran \otimes q) = w_d(q) = \wt{w_d}(\lan 1 \ran + \lan 2\ran \otimes q).$$
	Alternatively, one could also give a more direct definition of modified Stiefel-Whitney classes not depending on the parity of $q$ by setting $\wt{w_d}(q)$ as $w_d(q)$ if $d$ is odd and as $w_d(q) + \{2\} w_{d - 1}(q)$ if $d$ is even.
\end{example}

Finally, we recall another kind of invariants.

\begin{example}[Witt-ring invariants]
\label{pfisterInvariant}
The image of an $n$-dimensional quadratic form in the Witt ring $G$ yields an invariant $\Inv^*(O_n, W)$. Since
the definition of invariants only makes use of the functor property, this concept makes sense, even though $G$ is not
a cycle module. Albeit of limited use in the setting of quadratic forms, the aforementioned invariant becomes a refreshing
source of invariants for groups $G$ embedding into $O_n$. Indeed, for Weyl groups $G$ of type $D_{2n}, E_7, E_8$, we
construct embeddings such that the restrictions become invariants with values in a suitable power of the fundamental ideal $I \ss W$. Since the Milnor morphism 
\begin{align*}
\fmil_n:\, \MK_n & \to I^n/I^{n + 1}\\
\{\a_1\} \cdots \{\a_n\} &\mapsto \lan\lan \a_1 \ran\ran \otimes \cdots \otimes \lan\lan \a_n \ran\ran
\end{align*}
with $\lan\lan a \ran\ran := \lan 1, -a\ran$ induces an isomorphism between mod 2 Milnor K-theory and the graded Witt ring \cite[Theorem 4.1]{OVV07}, we obtain elements in $\Inv^\ast(G, \MK_\ast)$.
\end{example}

\medbreak

\subsection{A technical lemma}
The following technical lemma simplifies the computations of invariants.

\smallbreak

\begin{lemma}
\label{orbitSum1}
Let $R$ be a commutative ring, $I$ a finite index set, $M$ an $R$-module and $G$ a finite group acting on $I$.
The operation of $G$ on $I$ induces an operation of $G$ on the $R$-module $N:= \oplus_{i \in I}M$ by permutation of coordinates.
Let $I = I_1\udot I_2\udot\cdots\udot I_k$ be its 	orbit decomposition. Then, $N^G\cong \oplus_{i \le k} N_i$, where for $i \le k$, 
$$N_i:= \Big\{\sum_{j \in I_i}\iota_j(m):\, m \in M\Big\}\cong M.$$
Here, $\iota_j:\, M \to N$ denotes the inclusion along the $j$th coordinate.
\end{lemma}

\begin{proof}
	Since $(\sum_{j \ne i}N_j) \cap N_i = \{0\}$ and $\oplus_{i \le k}N_i\ss N^G$ hold for every $i$, it remains to show that the
	$N_i$ generate $N^G$. To prove this, note that any $x \in N$ can be written uniquely as $x = \sum_{i \in I}\iota_i(m_i)$ for
	certain $m_i \in M$. We prove by induction on the number of non-zero $m_i$ that any $x \in N^G$ lies in the module
	generated by the $N_i$. We may suppose $I = [1;|I|]$, $m_1 \ne 0$ and denote by $I_1$ the orbit containing 1.
	Now, comparing the $g(1)$th entry of $x$ and of $g\ldot x$ yields that $m_{g(1)} = m_1$ for every $g \in G$.
	In particular, we can split of a sum $\sum_{j \in I_1}\iota_j(m_j) = \sum_{j \in I_1}\iota_j(m_1) \in N_1$ from $x$. Applying induction to $x-\sum_{j \in I_1}\iota_j(m_1)$ concludes the proof.
\end{proof}

\smallbreak

In particular, Lemma \ref{orbitSum1} yields the following orbit decomposition.
\begin{corollary}
\label{orbitSum}
Let $R_*$ be a commutative, graded ring, $I^1, \dots, I^{r}$ be finite index sets, $M_*$ be a graded $R_*$-module
and $G$ a finite group acting on each of the $I^{\ell}$. The operation of $G$ on the $I^{\ell}$ induces an operation of
$G$ on the graded $R_*$-module $N_*:= \oplus_{\ell \le r}\oplus_{I^{\ell}}M_{* - d_\ell}$, where the $d_\ell$ are certain
non-negative integers. Let $I^{\ell} = I^{\ell}_1\udot I_2^{\ell}\udot\cdots\udot I^{\ell}_{n_\ell}$ be the orbit decomposition.
Then, $N^G\cong \oplus_{\ell \le r}\oplus_{i \le n_\ell} N_{\ell, i}$, where for $ \ell \le r$, $ i \le n_\ell$, we put 
$$(N_{\ell, i})_*:= \Big\{\sum_{j \in I^{\ell}_i}\iota_j(m)\;:\,\; m \in M_{* - d_\ell}\Big\}\cong M_{*-d_\ell}.$$
\end{corollary}

\begin{center}
{\bf\large
Part II: Computation of the invariants of irreducible Weyl groups
}
\end{center}

\noindent
Throughout this part~$\gk$ denotes a field of characteristic not~$2$.
When we compute the invariants of an irreducible Weyl group~$W=W(\Sigma)$, where~$\Sigma$ is an
irreducible root system we assume also that the characteristic of~$\gk$ and the order of~$G$ are coprime.

\smallbreak

We use in the following the description of irreducible root systems given in Bourbaki~\cite[PLATES I-VIII]{LIE4-6} for irreducible root systems of type $\not=G_2$ (recall that for Weyl groups of type~$G_2$ we have already computed
the invariants in Section \ref{G2InvSubSect}). We have $\Sigma\ss\bigoplus\limits_{i\le n}e_i\Z [1/2]\ss\R^n$
for an appropriate $n$. Taking the tensor product $\gk\otimes_{\Z [1/2]}$
we get an embedding of $\Sigma$ into $\gk^n$, such that all $\alpha\in\Sigma$ are anisotropic for the standard
scalar product of $\gk^n$. Hence the associated reflections generate a finite subgroup of $O_n(\gk)$ which is
isomorphic to $G$. In the following we will identify $G$ with this subgroup of $O_n(\gk)$.
\smallbreak

We provide a family of elements $\{ x_i\}_{i \in I}\ss\Inv(G, \MK_\ast)$, forming a basis of $\Inv(G, \CM_\ast)$ for all cycle modules over $\gk$.
For this we have to show that given $k \in \Fields_{\gk}$ and an invariant $a \in \Inv_k^\ast(G, \CM_\ast)$, then there exist unique $c_i \in \CM_\ast(k)$ such that 
$$
a = \sum\limits_{i \in I}\res_{k/\gk}(x_i) c_i.
$$
To verify this claim, we may assume $k = \gk$ and let $e_1, \dots, e_n$ denote the standard basis elements of the $\gk$-vector space $\gk^n$.

If $a_1, \dots, a_n \in \Sigma$ are pairwise orthogonal, then $P(a_1, \dots, a_n)$ denotes the elementary 2-abelian subgroup generated by the
corresponding reflections $s_{a_1}, \dots, s_{a_n}$. For $1 \le i_1 < \cdots < i_l \le n$, we write $x_{a_{i_1}, \dots, a_{i_l}}$ for the invariant
$$
H^1( -, (\Z/2) \cdot s_{a_1}\times\dots\times (\Z/2) \cdot s_{a_n})\xrightarrow{\simeq}H^1( -, (\Z/2)^n)
\xrightarrow{x_{i_1, \dots, {i_l}}}\MK_l(\, -\, ), 
$$
see Corollary \ref{Z/2InvSec} for the definition of the invariant $x_{i_1, \dots, i_l}$.

\section{Weyl groups of type $A_n$}
The invariants of Weyl groups of type $A_n$ with values in $\MK_\ast$ are induced by the Stiefel-Whitney classes $\{w_i\}_i$,
see \cite[Part I, Sect.\ 25]{CohInv}. The proof carries over essentially verbatim to invariants with values in cycle modules $\CM_\ast$
with $\MK_\ast$-structure using the splitting principle in the form of Proposition~\ref{splitCorollary} and the computation of
$\Inv((\Z/2)^n, \CM_\ast)$ in Corollary \ref{Z/2InvSec}. The result is as follows. Here, we identify $H^1(k, S_n)$ with the set of
isomorphism classes of \'etale algebras of dimension $n$ over $k$, and denote for such an algebra $E$ by $q_E$ its trace form.

\begin{proposition}
	Let $n \ge 1$. Then, $\Inv(S_n, \CM_\ast)$ is completely decomposable with basis $\{E\mapsto w_i(q_E)\}_{i \le \lfloor n /2 \rfloor}$.
\end{proposition}


\goodbreak
\section{Weyl groups of type $B_n/C_n$.}
\label{BnSubSect}
First, we note that the Weyl group $W(C_n)$ is isomorphic to the Weyl group $W(B_n)$. Hence, determining the invariants for $W(B_n)$ will also yield the determinants for $W(C_n)$.

\subsection{Invariants of $B_2$}
\label{b2Sec}
First, we consider $W(B_2)$, which is isomorphic to the dihedral group of order $8$. In particular, $G:= W(B_2) = \lan\sigma, \tau\ran\ss S_4$ admits the permutation representation defined by
\begin{align*}
\sigma& = 
\begin{pmatrix}
1 & 2& 3 & 4 \\
2 & 3& 4 & 1
\end{pmatrix}, \;\phantom{aaaa}
\tau = 
\begin{pmatrix}
1 & 2& 3& 4 \\
3 & 4& 1 & 2
\end{pmatrix}.
\end{align*}
\smallbreak

\smallbreak

Considering $G$ as orthogonal reflection group over $\gk$ yields an embedding $\phi:\, G\ss O_2$ of algebraic groups over $k_0$ given by
\begin{align*}
\sigma\mapsto 
\begin{pmatrix}
0 & -1 \\
1 & 0\\
\end{pmatrix}
\text{, }
\tau\mapsto 
\begin{pmatrix}
0 & 1 \\
1 & 0\\
\end{pmatrix}. 
\end{align*}
Now, $\phi$ determines
an action of $G$ on $\gk [X, Y]$ given by ${^\sigma X} = Y$, ${^\sigma Y} = -X$, ${^{\tau} X} = Y$, ${^{\tau} Y} = X$.
In particular, $\gk [X, Y]^G = \gk [X^2 + Y^2, X^2Y^2]\cong\gk [A, B]$, where $A:= X^2 + Y^2$, $B:= 4X^2Y^2$.
Fix the notation $E:= \gk(X, Y)$, $K:=\gk (X^2 + Y^2, X^2Y^2)$. Now, the group $G$ acts freely on the open subscheme
\begin{align*}
	U:= D \big(XY(X - Y)(X + Y)\big) = D(X^2Y^2(X^2 - Y^2)^2)\ss \A^2,
\end{align*} 
where for a polynomial $f$, we denote by $D(f)\ss \A^2$ the open subset given by $f \ne 0$. 

By \cite[Part I, Thm.~12.3]{CohInv} or \cite[Thm.\ 3.7]{GiHi19}, the evaluation at the versal torsor $\Spec(E) \to \Spec(K)$
yields an injection $\Inv(G, \CM_\ast) \to \CMu \ast(U/G)$. To check that this
map is also surjective, we first compute $\CMu \ast(U/G)$. An explicit computation yields
\begin{align*}
	U/G&\cong \Spec(\gk [X, Y, X^{-2}Y^{-2}(X^2 - Y^2)^{-2}]^G)\\
	& = \Spec(\gk [X^2 + Y^2, X^2Y^2, X^{-2}Y^{-2}, (X^2 - Y^2)^{-2}])\\
	&\cong \Spec \big(\gk \big[A, B, B^{-1}, (B - A^2)^{-1}\big]\big), 
\end{align*}

\smallbreak

To compute $\CM_{\ast, \ms{unr}}(U/G)$, note that $V:= D(A)\ss U/G$ is isomorphic to the spectrum of
\begin{align*}
	k_0 \big[A, B, B^{-1}, A^{-1}, (B - A^2)^{-1}\big]	\cong k_0 \big[A, B',(B')^{-1}, A^{-1}, (B' - 1)^{-1}\big], 
\end{align*}
where the isomorphism is induced by mapping $B'$ to $B/A^2$. Now, by applying Lemma~\ref{cheeseLine} twice and homotopy invariance,
\begin{align*}
	\CMu{\ast}(V)&\cong\CM_\ast(\gk)\oplus \{B/A^2 - 1\}\CM_{\ast -1}(\gk)\oplus \{A\}\CM_{\ast -1}(\gk)\oplus\\
	&\phantom=\oplus \{B\}\CM_{\ast -1}(\gk)\oplus \{A\}\{B/A^2 - 1\}\CM_{\ast -2}(\gk)\\
	&\phantom=\oplus \{A\}\{B\}\CM_{\ast -2}(\gk).
\end{align*}
$\CM_{\ast, \ms{unr}}(U/G)$ can be computed as the kernel of the boundary $\partial = \partial^A_{(A)}:\,\CM_\ast(V) \to \CM_{\ast -1}(\G_m)$.
Thus, for every $t \in \CM_\ast(\gk)$,
\begin{align*}
\partial(t)& = 0, \\
\partial(\{B/A^2 - 1\}t)& = \partial(\{B - A^2\}t)
	 = \{B\}\partial(t)
	 = 0, \\
\partial(\{B\}t)&  
	 = \{B\}\partial(t)
	 = 0, \\
\partial(\{A\}t)& = t, \\
\partial(\{A\}\{B/A^2 - 1\}t)& = \partial(\{A\}\{B - A^2\}t)
	 = \{B\}\partial(\{A\}t)
	 = \{B\}t.\\
\partial(\{A\}\{B\}t)& 
	 = \{B\}\partial(\{A\}t)
	 = \{B\}t.
\end{align*}
Writing $\CM_\ast$ short for $\CM_\ast(\gk)$, we conclude that $\CM_{\ast, \ms{unr}}(U/G)$ is given by
\begin{align*}
	&\CM_\ast\oplus \{B - A^2\}\CM_{\ast -1}\oplus \{B\}\CM_{\ast -1}\oplus \{A\}\{B(B - A^2)\}\CM_{\ast -2}\\
	&\cong\CM_\ast\oplus \{B - A^2\}\CM_{\ast -1}\oplus \{B\}\CM_{\ast -1}\oplus \{A\}\{B - A^2\}\CM_{\ast -2}.
\end{align*}
It remains to construct invariants mapping to the three non-constant basis elements of $\CMu*(U/G)$.
Pulling back $w_1, w_2 \in \Inv(O_2, \MK_\ast)$ along the embedding $\phi$ gives invariants in $\Inv(G, \MK_\ast)$ that -- by abuse of
notation -- we again denote by $w_1, w_2$. We first compute the value $w_1(E/K)$ of $w_1$ at the versal torsor $E/K$
constructed above. To do this, we note that the determinant of $\phi(\sigma^i\tau)$ is $-1$, while the determinant of
$\phi(\sigma^i)$ is 1. Now, $XY(X^2 - Y^2) \in E$ maps to its negative by each reflection and is fixed by all the $\sigma^i$. Thus, $w_1(E/K) = \{X^2Y^2(X^2 - Y^2)^2\} = \{B(A^2 -B)\}$.

\medbreak

Another invariant comes from the embedding $G\ss S_4$. We may define $v_1:=\ms{res}^G_{S_4}(\wt{w_1})$. Again, we
compute $v_1(E/K)$. We note that $\wt{w_1} \in \Inv^1(S_4, \MK_\ast)$ may be computed as follows. Start with an arbitrary
$x \in H^1(k, S_4)$; then $\wt{w_1}(x) = \ms{sgn}_*(x) \in H^1(k, \Z/2)\cong k^\times/k^{\times2}\cong\MK_1(k)$. The kernel
of $\ms{sgn}$ consists exactly of the elements $\{id, \tau, \sigma^2, \sigma^2\tau\}$ with $\sigma, \tau$ as above. Since
$XY$ is fixed by this kernel and is mapped to its negative by $\sigma$, the value of $v_1$ at the versal torsor is
$\{X^2Y^2\} = \{B\}$. Consequently, it remains to find an invariant mapping to the basis $\{A\}\{B^2 - A\}$ of $\CM_{\ast, \ms{unr}}(U/G)$.

\smallbreak

Finally, we compute the value of $w_2 \in \Inv^2(G, \MK_\ast)$ at $E/K$. First consider the elementary abelian 2-subgroup
generated by reflections $P:= \lan\tau, \tau'\ran$, where $\tau' = \sigma^2\tau$. Thus, 
\begin{align*}
\phi(\tau) = \begin{pmatrix}
0 & 1\\ 
1 &0 
\end{pmatrix}, 
\phi(\tau') = \begin{pmatrix} 
0 &-1 \\
-1 &0 
\end{pmatrix}.
\end{align*}
Recalling that the action of $G$ on $E$ is defined via $\phi$, we now consider the versal $P$-torsor
$E/E^P = \gk (X, Y)/\gk(X^2 + Y^2, XY)$. Then, $\tau \in P = \ms{Gal}(E/E^P)$ acts via $\tau(X) = Y, \tau(Y) = X$
and $\tau'$ via $\tau'(X) = -Y, \tau'(Y) = -X$. Thus, this $(\Z/2)^2$-torsor over $E^P$ is equivalently described by the pair
$((X-Y)^2, (X + Y)^2) \in ((E^P)^\times/(E^P)^{\times 2})^2$. We conclude that the value of $\ms{res}^P_{O_4}w_2$ at this $P$-torsor is $\{(X-Y)^2\}\{(X + Y)^2\} \in \MK_2(E^P)$.

\smallbreak

By the computations above, the value of $\ms{res}^G_{O_4}(w_2)$ at $E/K$ is of the form 
$$\a_1 + \{B - A^2\}\a_2 + \{A\}\a_3 + \{B\}\{B(B - A^2)\}\a_4 \in \MK_2(K)$$
for some $\a_1 \in \MK_2(k_0)$, $\a_2, \a_3 \in \MK_1(k_0)$, $\a_4 \in \MK_0(k_0)$. Now, consider the diagram
$$
\xymatrix{
 H^1(K, G)\ar[r]^-{w_2}\ar[d]_{\ms{res}^{E^P}_K(E)}& \MK_2(K) \ar[d]\\
H^1(E^P, G)\ar[r]^-{w_2}& \MK_2(E^P)\\
H^1(E^P, P).\ar[u]^{\ms{ind}^G_P}&
}
$$
The square commutes by the definition of invariants. Denote by $E \in H^1(K, G)$ the $G$-torsor $E/K$ and by
$F \in H^1(E^P, P)$ the $P$-torsor $E/E^P$. Interpreting the torsors as cocycles yields 
$$\ms{ind}_P^G(F) = \ms{res}^{E^P}_K(E) \in H^1(E^P, G).$$
Observing that $XY$ is a square in $E^P$, this means
$$
\{(X-Y)^2\}\{(X + Y)^2\} = \a_1 + \{B - A^2\}\a_2 + \{A\}\{A^2 -B\}\a_4.
$$
Applying the identity $\{\b\}\{\b'\} = \{\b + \b'\}\{-\b\b'\}$ to the left-hand side gives $\{2A\}\{B - A^2\}$, so that we may choose $\a_1 = 0$, $\a_2 = \{2\}$ and $\a_4 = 1$.
We conclude that the injection $\Inv (G, \CM_\ast) \to \CM_{\ast, \ms{unr}}(U/G)$ is surjective. 
This finishes the computation of $\Inv (G, \CM_\ast)$ and we obtain the following.

\begin{proposition}
The invariants $\Inv (W(B_2), \CM_\ast)$ are completely decomposable with basis consisting of the invariants $\{1, v_1, w_1, w_2\}$.
\end{proposition}

We conclude this section with a corollary of the proof.
\begin{corollary}
\label{i2Corollary}
Let $P_1 = P(e_1, e_2)$ and $P_2 = P(e_1 - e_2, e_1 + e_2)$. Then, 
\begin{align*}
\ms{res}^{P_1}_{W(B_2)}(v_1)& = x_{\{e_1\}} + x_{\{e_2\}}, \\
\ms{res}^{P_1}_{W(B_2)}(w_1)& = x_{\{e_1\}} + x_{\{e_2\}}, \\
\ms{res}^{P_1}_{W(B_2)}(w_2)& = x_{\{e_1, e_2\}},
\end{align*}
and 
\begin{align*}
\ms{res}^{P_2}_{W(B_2)}(v_1)& = 0, \\
\ms{res}^{P_2}_{W(B_2)}(w_1)& = x_{\{e_1 - e_2\}} + x_{\{e_1 + e_2\}}, \\
\ms{res}^{P_2}_{W(B_2)}(w_2)& = x_{\{e_1 + e_2, e_1 - e_2\}} + \{2\} \cdot(x_{\{e_1 - e_2\}} + x_{\{e_1 + e_2\}}).
\end{align*}
\end{corollary}

\subsection{Invariants of $B_n$}
After dealing with the case $n = 2$, we now compute the invariants of Weyl groups of type $B_n$ for general $n$. The root system $B_n$ is the disjoint union
$\Delta_1\udot \Delta_2\ss \R^n$, where $\Delta_1 = \{\pm e_i:\,1 \le i \le n\}$ are the short roots and
$\Delta_2 = \{\pm e_i \pm e_j:\, 1 \le i < j \le n\}$ are the long roots. This root system induces an orthogonal
reflection group over any $k_0$ satisfying the above requirements. Furthermore, $W(B_n)\cong S_n\ltimes (\Z/2)^n$
as abstract groups. Put $m:= [n/2]$ and for $ i \le m$ define $a_i:= e_{2i - 1} - e_{2i}$ and $b_i:= e_{2i - 1} + e_{2i}$.
For each $ L \le m$ the elements of $X_L:= \{a_1, b_1, \dots, a_L, b_L, e_{2L + 1}, e_{2L + 2}, \dots, e_n\}$ are mutually
orthogonal. Defining $P_L:= P(X_L)$, we prove by induction on $m$ that $\Omega(G) = \{[P_0], \dots, [P_m]\}$.

The claim is clear for $n = 2$. In the general case, let $P$ be any maximal elementary abelian $2$-subgroup generated
by reflections. First assume that $P$ contains a short root, say $e_n$. Now, observe that
$\lan e_n\ran^{\perp} \cap B_n = B_{n - 1}$ and use induction. If $P$ contains a long root, we may assume this root to
be $a_1$. Then, $\lan a_1\ran^\perp \cap B_n = \{\pm b_1\}\cup B_{n - 2}$, where we consider $B_{n - 2}$ to be embedded
in $\R^n$ using the last $n - 2$ coordinates. In particular, we may again use the induction hypothesis.

\medbreak

To determine $\Inv (B_n, \CM_\ast)$, we introduce additional pieces of notation. We denote $P_L$-torsors over a field
$k$ by $(\a_1, \b_1, \dots, \a_L, \b_L, \epsilon_{2L + 1}, \dots, \epsilon_n) \in (k^\times/k^{\times 2})^n$. From the
$(\Z/2)^n$-section, we know that $\Inv(P_L, \CM_\ast)$ is completely decomposable with basis $\{x_I\}_{I\ss [1;n]}$.
Since this parameterization is inconvenient in the present setting, we change the index set by putting
\begin{align*}
	\La^d_L:= \{(A, B, C, E)\ss [1;L]^3\times[2L + 1;n]:\,& A, B, C \text{ pw.~disjoint}, \\
&|A| + |B| + 2|C| + |E| = d\}.
\end{align*}
We reindex the basis of $\Inv (P_L, \CM_\ast)$ by defining for every $(A, B, C, E) \in \La^d_L$:
\begin{align*}
x_{A, B, C, E}^L:\, H^1(k, P_L)& \to \MK_\ast(k)\\
(\a_1, \b_1, \dots, \a_L, \b_L, \epsilon_{2L + 1}, \dots \epsilon_n)&\mapsto
\prod_{a \in A}\{\a_a\} \prod_{b \in B}\{\b_b\} \prod_{c \in C}\{\a_c\}\{\b_c\} \prod_{e \in E}\{\epsilon_e\}.
\end{align*}
In the same spirit, we also write
$$
P(A, B, C, E):= P(\{a_p\}_{p \in A}\cup \{b_q\}_{q \in B}\cup \{a_r, b_r\}_{r \in C}\cup \{e_s\}_{s \in E}).
$$

\medbreak

\noindent
For $d \le n$, we now construct the specific $W(B_n)$-invariant
$$u_d:= \rho^*(\wt{w_d}) \in \Inv^d(W(B_n), \CM_\ast),$$
where $\wt{w_d} \in \Inv^d(S_n, \MK_\ast)$ denotes the $d$th modified Stiefel-Whitney class and 
$
\rho:\, W(B_n)\cong S_n\ltimes (\Z/2)^n \to S_n
$
is the canonical projection. Then, the map $W(B_n) \to S_n$ sends both $s_{a_i}, s_{b_i}$ to $(2i - 1, 2i)$ and
$s_{e_i}$ to the neutral element. Let $k \in \Fk$ and $(\a_1, \b_1, \dots, \a_L, \b_L, \epsilon_{2L + 1}, \dots, \epsilon_n)$
be a $P_L$-torsor over $k$. Using Example \ref{abQuadratic2} and $\{2\}\{2\} = 0$, gives that the value of the total modified Stiefel-Whitney class at this torsor is $\prod_{i \le L }(1 + \{\a_i\b_i\})$. Hence,
\begin{equation}
\label{bnUd}
\ms{res}^{P_L}_{W(B_n)}(u_d) = \sum_{\substack{(A, B, \es, \es) \in \La^d_L}} x^L_{A, B, \es, \es}.
\end{equation}
Next, we construct an invariant $v_d$ such that 
\begin{equation}
\label{bnVdEq}
\ms{res}^{P_L}_{W(B_n)}(v_d) = \sum_{(\es, \es, C, E) \in \La^d_L} x^L_{\es, \es, C, E}
\end{equation}
To that end, we note that $W(B_n)$ embeds into $S_{2n}$ via $\sigma \prod_{i \in I}s_{e_i}\mapsto \sigma \cdot(\sigma + n) \prod_{i \in I}(i, i + n)$, 
where $I\ss [1;n]$, $\sigma \in S_n$ and $\sigma + n \in S_{2n}$ is given by
\begin{align*}
k\mapsto \begin{cases}
k &\text{ if }k \le n,\\
n + \sigma(k - n) &\text{ if }k > n.
\end{cases}
\end{align*}
We define the modified Stiefel-Whitney invariants $\wt{w_d} \in \Inv^d(S_{2n}, \MK_\ast)$ as before and put $v_d' := \ms{res}^{W(B_n)}_{S_{2n}}(\wt{w_d}) \in \Inv^d(W(B_n), \MK_\ast)$ for $d \le n$. Then, we define $v_d$ recursively, by setting $v_0 := 0$ and then
$$v_d := v_d' + \sum_{k \le d - 1}u_{d - k} v_k.$$
To show that the so-defined invariant satisfies \eqref{bnVdEq}, we first note that already when restricting $v_d'$ to $P_L$, we obtain an agreement with the right-hand side of \eqref{bnVd} up to mixed lower-order expressions.

\begin{lemma}
	\label{bnVdLem}
\begin{equation}
\label{bnVd}
	\ms{res}^{P_L}_{W(B_n)}(v_d') =\hspace{-.2cm} \sum_{(\es, \es, C, E) \in \La^d_L} \hspace{-.2cm}x^L_{\es, \es, C, E} + \sum_{k \le d - 1}\{-1\}^{d - k}\hspace{-.2cm}\sum_{(A, B, C, E) \in \La^k_L}\hspace{-.2cm} x^L_{A, B, C, E}
\end{equation}
\end{lemma}

\begin{proof}
	Observe that the map $W(B_n) \to S_{2n}$ sends $s_{e_i} \mapsto (i, i + n)$ and 
\begin{align*} 
	s_{a_i}\mapsto(2i - 1, 2i)(2i - 1 + n, 2i + n), \; 
	s_{b_i}\mapsto (2i - 1, 2i + n)(2i, 2i - 1 + n)
\end{align*} 
Hence, by Lemma \ref{pfisterLemma}, the composition $P_L \to W(B_n) \to S_{2n} \to O_{2n}$ maps a $P_L$-torsor to the quadratic form 
$$\lan\lan - \a_1, - \b_1 \ran\ran\oplus\dots\oplus \lan\lan - \a_L, - \b_L \ran\ran\oplus \lan2, 2\epsilon_{2L + 1}, \dots, 2, 2\epsilon_n\ran.$$
We claim that the total modified Stiefel-Whitney class evaluated at this quadratic form equals
	\begin{align}
		\label{vdswEq}
	\prod_{i \le L}(1 + \{-1\}(\{\a_i\} + \{\b_i\}) + \{\a_i\}\{\b_i\})\prod_{ 2L + 1\le i \le n}(1 + \{\epsilon_i\}).
\end{align}
To see this, we compute
	it suffices to check that $w(\lan2\ran\otimes \lan\lan\a, \b\ran\ran) = 1 +\{-1\}\{-1\} + \{\a\}\{\b\}$. To see this, we compute 
	\begin{align*}
		w( \lan2\ran\otimes\lan\lan-\a, -\b\ran\ran) &= (1 + \{2\})(1 + \{2\a\})(1 + \{2\b\})(1 + \{-2\b\} + \{-\a\})\\
		&= (1 + \{\a\} + \{2\}\{\a\}) (1 + \{\a\} + \{2\b\}\{-\a\})\\
		&= 1 + \{\a\}\{\a\} + \{2\}\{\a\} + \{2\b\}\{-\a\}\\
		&= 1 +\{-1\}\{\a\} + \{-1\}\{\b\} + \{\a\}\{\b\}.
	\end{align*}
	Thus, translating \eqref{vdswEq} into the new notation, we obtain that
	$$\ms{res}^{P_L}_{W(B_n)}(v_d') = \sum_{(\es, \es, C, E) \in \La^d_L} x^L_{\es, \es, C, E} + \sum_{k \le d - 1}\{-1\}^{d - k}\hspace{-.5cm}\sum_{(A, B, C, E) \in \La^k_L} x^L_{A, B, C, E}.\qedhere$$
\end{proof}

In light of Lemma \ref{bnVd}, to establish \eqref{bnVdEq}, it remains to understand the product structure between $u_{d - k}$ and $v_k$. To that end, we restrict the products to $P_L$.

\begin{lemma}
\label{bnLemma1}
We have 
$$
\sum_{(A, B, \es, \es) \in \La_L^d}x^L_{A, B, \es, \es} \sum_{(\es, \es, C, E) \in \La_L^f}x^L_{\es, \es, C, E} = \sum_{\substack{(A, B, C, E) \in \La^{d + f}\\2|C| + |E| = f}}x^L_{A, B, C, E}.
$$
\end{lemma}

\begin{proof}
First, since $x^L_{A, B, \es, \es} x^L_{\es, \es, C, E} = \{-1\}^{|A \cap C| + |B \cap C|} x^L_{A - C, B - C, C, E}$,
\begin{align*}
	&\sum_{(A, B, \es, \es) \in \La_L^d}x^L_{A, B, \es, \es} \sum_{(\es, \es, C, E) \in \La_L^f}x^L_{\es, \es, C, E}\\
	&\qquad= \sum_{k \ge 0}\sum_{\substack{(A, B, \es, \es) \in \La_L^d\\(\es, \es, C, E) \in \La_L^f\\|A \cap C| + |B \cap C| = k}} \{-1\}^kx^L_{A - C, B - C, C, E} \\
	&\qquad= \sum_{\substack{(A, B, C, E) \in \La_L^{d + f}\\2|C| + |E| = f }} x^L_{A, B, C, E} 
	+ \sum_{k \ge 1}
	 \sum_{\substack{(A, B, \es, \es) \in \La_L^d\\(\es, \es, C, E) \in \La_L^f\\|A \cap C| + |B \cap C| = k}} \{-1\}^kx^L_{A - C, B - C, C, E}.
\end{align*}
To show that the second sum vanishes, fix $k \ge 1$ and $(A', B', C, E) \in \La_L^{d + f - k}$. Then, define 
	\begin{align*}
		S&:=\{(A, B):\, (A, B, \es, \es) \in \La_L^d \text{ and }A - C = A' \text{ and }B - C = B'\}\\
		& = \{(A'\cup U, B'\cup V):\, U, V\ss C\text{ and } U \cap V = \es \text{ and }|U| + |V| = k\}.
	\end{align*}
Using this description, we conclude $|S| = 2^k\Big(\begin{array}{c} |C| \\ k\end{array}\Big)$. Since $k \ge 1$, this is even and we obtain the desired vanishing of the second sum.
\end{proof}
In the rest of this section, we show that $\Inv (W(B_n), \CM_\ast)$ is completely decomposable and that the products $\{u_{d - r} v_r\}_{\substack{\max(0, 2d - n) \le r \le d\\ d \le n}}$ yield a basis.
\smallbreak

Before determining the structure of $\Inv (W(B_n), \CM_\ast)$, it is helpful to know something about the image of the restriction maps
$\Inv (W(B_n), \CM_\ast) \to \Inv (P_L, \CM_\ast)$. Let $d, k, \ell, L$ be non-negative integers, $L \le m$. Then, the invariant
$$
\phi^d_{L, k, \ell}:=\sum_{\substack{(A, B, C, E) \in \La^d_L\\|C| = k, |E| = \ell}}x^L_{A, B, C, E}
$$
is non-trivial if and only if there exists $(A, B, C, E) \in \La_L^d$ with $|C| = k$ and $|E| = \ell$.

\begin{lemma}
	\label{prevLem}
The image of the restriction map $\Inv (W(B_n), \CM_\ast) \to \Inv (P_L, \CM_\ast)$ is contained in the free submodule with basis 
$$
 \big\{\phi^d_{L, k, \ell}:\, 2k + \ell \le d \le n, \; 2(d - k-\ell) \le 2L \le n - \ell\big\}.
$$
\end{lemma}

\begin{proof}
Let us first show that $\phi^d_{L, k, \ell} \ne 0$ iff $2k + \ell \le d \le n$ and $2(d - k-\ell) \le 2L \le n - \ell$. First, the conditions $2k + \ell \le d$ and $2L + \ell \le n$ are necessary. Furthermore, from the pairwise disjointness of $A, B, C$, we conclude $|A| + |B| + |C| \le L$. This is equivalent to $d - (2k + \ell) + k \le L$. Thus, $d - k-\ell \le L$ is also necessary. To check sufficiency, suppose, we are given $L, k, \ell, d$ satisfying the restrictions. Then, $([1;d - \ell-2k], \es, [d - \ell-2k + 1;d - \ell-k], [2L + 1;2L + \ell]) \in \La^d_L$. Thus, $\phi^d_{L, k, \ell} \ne 0$. Next, we check that the image of the restriction map is indeed contained in the submodule generated by the $\phi^d_{L, k, \ell} \CM_\ast(\gk)$.

\smallbreak

Observe that all of the following elements normalize $P_L$:
	$$\{s_{e_{2i - 1} - e_{2j-1}} s_{e_{2i} - e_{2j}}\}_{i, j \le L},\qquad \{s_{e_i - e_j}\}_{i, j \ge 2L + 1}\quad \text{and}\quad \{s_{e_{2i}}\}_{i \le L}.$$
Let $N_L\ss N_{W(B_n)}(P_L)$ be the subgroup generated by these elements. We claim that $N_L$ permutes the $x^L_{A, B, C, E}$. Applying $s_{e_{2i - 1} - e_{2j-1}} s_{e_{2i} - e_{2j}}$ for $i, j \le L$ to a $P_L$-torsor
$$
(\a_1, \b_1, \dots, \a_L, \b_L, \epsilon_{2L + 1}, \dots, \epsilon_n)
$$
	interchanges $\a_i \lra \a_j$ and $\b_i \lra \b_j$. Thus, $x^L_{A, B, C, E}$ maps to $x^L_{A', B', C', E}$ where $A'/B'/C'$ is obtained from $A/B/C$ by applying the transposition $(i, j)$ to the respective sets. Similarly, we see that swapping the $i$th and the $j$th coordinate for $i, j \ge 2L + 1$ maps $x^L_{A, B, C, E}$ to $x^L_{A, B, C, E'}$ where $E'$ is obtained from $E$ by applying to it the transposition $(i, j)$. Finally, changing the $(2i)$th sign maps $x^L_{A, B, C, E}$ to $x^L_{A', B', C, E}$ where $A' = (A - \{i\})\cup (B \cap \{i\})$ and $B' = (B - \{i\})\cup (A \cap\{i\})$. That is,  if $i \in A$ we remove it from $A$ and put it into $B$ and vice versa.

\smallbreak

Iteratively applying these operations to an arbitrary $(A_0, B_0, C_0, E_0) \in \La_L^d$ shows that its orbit under $N_L$ equals $ \{(A, B, C, E) \in \La_L^d\;:\,\; |C| = |C_0|, |E| = |E_0|\}$. Now, the lemma follows from Corollary \ref{orbitSum}.
\end{proof}

By Proposition \ref{splitCorollary}, the injection $\Inv (W(B_n), \CM_\ast) \to \prod_{L \le m}\Inv (P_L, \CM_\ast)$ has its image inside $\prod_{L \le m}\Inv (P_L, \CM_\ast)^{N_L}$ and Lemma \ref{prevLem} gives a good description of this object. However, this map is not surjective. One reason is the following: If an element $(z_L)_L$ of the right hand side comes from a $W(B_n)$-invariant, then certainly the restrictions of $z_L$ and $z_{L'}$ to $P_L \cap P_{L'}$ must coincide. To address this, we prove the following refined lemma.

\smallbreak

\begin{lemma}
\label{refinedBn}
	The image of $\Inv (W(B_n), \CM_\ast) \to \prod_{L \le m}\Inv (P_L, \CM_\ast)$ lies in the subgroup generated by $\{s \cdot M_{*-|s|}(k_0):\,s \in S\}$, where
$$
S:= \Big\{ \Big(\sum_{2k + \ell = r} \phi_{L, k, \ell}^d\Big)_L:\, \max(0, 2d - n) \le r \le d \le n\Big\}\ss \prod_{L \le m}\Inv (P_L, \MK_\ast).
$$
\end{lemma}

\begin{proof}
	Let $\wt z \in \Inv (W(B_n), \CM_\ast)$ be a homogeneous invariant and $z = (z_L)_L \in \prod_{L \le m}\Inv (P_L, \CM_\ast)$ 	be the image of $\wt z$ under the restriction maps. By Lemma \ref{prevLem}, $z = \big(\sum_{d, k, \ell}\phi_{L, k, \ell}^dm_{L, d, k, \ell}\big)_L$ for some $m_{L, d, k, \ell} \in \CM_{\ast -d}(\gk)$, where the sums are over all those $d, k, \ell$ such that $\phi_{L, k, \ell}^d \ne 0$. 

\smallbreak

First goal, we show that $m_{L, d, k, \ell}$ is independent of $L$ in the sense that $m_{L, d, k, \ell} = m_{L', d, k, \ell}$, if $\phi^d_{L, k, \ell} \ne 0$ and $\phi^d_{L', k, \ell} \ne 0$. We then denote by $m_{d, k, \ell}$ the common value. Observe that $(A_0, B_0, C_0, E_0) \in \La^d_{L'} \cap\La^d_L$, where 
	$$
(A_0, B_0, C_0, E_0):=([1;d - \ell-2k], \es, [d - \ell-2k + 1;d - \ell-k], [n - \ell + 1;n]).
$$
 Hence, since $z$ comes from an invariant of $W(B_n)$, 
$$
\ms{res}^{P(A_0, B_0, C_0, E_0)}_{P_L}(z_L) = \ms{res}^{P(A_0, B_0, C_0, E_0)}_{P_{L'}}(z_{L'}).
$$
Comparing coefficients of $x_{A_0, B_0, C_0, E_0}$-components on both sides yields that $m_{L, d, k, \ell} = m_{L', d, k, \ell}$.

\smallbreak

Now, let us have a look at the second obstruction. We want to prove $m_{d, k, \ell} = m_{d, k', \ell'}$, if $2k + \ell = 2k' + \ell'$ and if there exist $L, L'$ such that $\phi^d_{L', k', \ell'} \ne 0$ and $\phi^d_{L, k, \ell} \ne 0$. It suffices to prove this in the case $k' - k = 1$. Since there exist $L, L'$ satisfying $\phi^d_{L', k', \ell'}, \phi^d_{L, k, \ell} \ne 0$, we can choose some $L$ such that $\phi^d_{L + 1, k', \ell'}, \phi^d_{L, k, \ell} \ne 0$. 

\smallbreak

	Let $y$ be the restriction of $\wt z$ to $P([1;d - \ell-2k], \es, [L - k + 1;L], [2L + 3;2L + \ell])\times W(B_2)$, where $B_2$ is embedded via the $(2L + 1)$th and the $(2L + 2)$th coordinates. By Proposition \ref{productLem}, 
$$
y = \sum_{\substack{A\ss[1;d - \ell-2k]\\ C\ss [L - k + 1;L]\\ E\ss[2L + 3;2L + \ell]}}x^L_{A, \es, C, E} y_{A, C, E}
$$
for uniquely determined $y_{A, C, E} \in \Inv^{\ast -|A|-2|C|-|E|}(W(B_2), \CM_\ast)$. Furthermore, by the results of Section \ref{b2Sec}, 
$$
y_{A, C, E} = m^{(0)}_{A, C, E} + w_1m^{(1a)}_{A, C, E} + v_1m^{(1b)}_{A, C, E} + w_2m^{(2)}_{A, C, E}
$$
for uniquely determined
$$
m^{(0)}_{A, C, E} \in \CM_{\ast -|A|-2|C|-|E|}(k_0), \; m^{(1a)}_{A, C, E}, m^{(1b)}_{A, C, E} \in \CM_{\ast -|A|-2|C|-|E|-1}(k_0)
$$
and
$$
m^{(2)}_{A, C, E} \in \CM_{\ast -|A|-2|C|-|E|-2}(k_0).
$$
Restricting $y$ further to $P([1;d - \ell-2k], \es, [L - k + 1;L], [2L + 1;2L + \ell])$ and considering the $x_{[1;d - 2k-\ell], \es, [L - k + 1;L], [2L + 1;2L + \ell]}$-component, Corollary \ref{i2Corollary} yields that
$$
m_{d, k, \ell} = m^{(2)}_{([1;d - \ell-2k], [L - k + 1;L], [2L + 3;2L + \ell])}.
$$
On the other hand, restricting $y$ to $P([1;d - \ell-2k], \es, [L - k + 1;L + 1], [2L + 3;2L + \ell])$ and considering the $x_{[1;d - 2k-\ell], \es, [L - k + 1;L + 1], [2L + 3;2L + \ell]}$-component, we obtain from Corollary~\ref{i2Corollary} that
$$
m_{d, k', \ell'} = m^{(2)}_{([1;d - \ell-2k], [L - k + 1;L], [2L + 3;2L + \ell])}.
$$
This proves the lemma.
\end{proof}

From Lemma \ref{bnLemma1}, we deduce the following decomposition of $\Inv (W(B_n), \CM_\ast)$.

\smallbreak

\begin{corollary}
The group $\Inv (W(B_n), \CM_\ast)$ is completely decomposable with basis
$$
 \big\{u_{d - r}v_r:\;\max(0, 2d - n) \le r \le d \le n\big\}.
$$
\end{corollary}


\goodbreak
\section{Weyl groups of type $F_4$.}
\label{F4SubSect}

\noindent
The root system $F_4$ is the disjoint union $\Delta_1\udot \Delta_2\udot \Delta_3\ss \R^4$ with short routes $\Delta_1:=\{\pm e_i\pm e_j:\, 1\le i < j \le 4\}$ and long roots
$$
\Delta_2:=\{\pm e_i:\, 1 \le i \le 4\}, \qquad \Delta_3:=\{1/2(\pm e_1\pm e_2\pm e_3\pm e_4)\}.
$$
Moreover, $\Omega(W(F_4)) = \{[P_0], [P_1], [P_2]\}$, where
$$P_0:= P(e_1, e_2, e_3, e_4), \qquad P_1:= P(a_1, b_1, e_3, e_4), \qquad P_2:= P(a_1, b_1, a_2, b_2)$$
\medbreak

Indeed, the set of long roots of $F_4$ is the root system $D_4$, which up to conjugacy has a unique maximal set of pairwise orthogonal vectors, namely $a_1, b_1, a_2, b_2$. On the other hand, if we have a maximal set of pairwise orthogonal roots containing a short root, say $e_4$, then $\lan e_4\ran^\perp \cap F_4 = B_3$. We have determined before that up to conjugacy $B_3$ contains two maximal sets of pairwise orthogonal roots; namely $\{e_1, e_2, e_3\}$ and $\{a_1, b_1, e_3\}$.

\smallbreak

Furthermore, the inclusion $P_2\ss W(B_4)\ss W(F_4)$ shows that the restriction map
$$
\Inv (W(F_4), \CM_\ast) \to \Inv (W(B_4), \CM_\ast)
$$
is injective. Recall that $\Inv (W(B_4), \CM_\ast)$ is a free $\CM_\ast(\gk)$-module with the basis 
$$\{1, u_1, v_1, u_2, v_1u_1, v_2, v_2u_1, v_3, v_4\}.$$
\medbreak

Before constructing specific invariants, we first point to another restriction in degree $2$. Since $\ms{res}^{P_2}_{W(F_4)}(v_1) = \ms{res}^{P_2}_{W(F_4)}(v_3) = 0$, the image of the restriction $\ms{res}^{P_2}_{W(F_4)}$ is contained in the free submodule $S\ss\Inv^*(P_2, \CM_\ast)$ with basis $\{1, y_1, y_2, y_2', y_3, y_4\}$, where $y_1 = \res_{W(B_4)}^{P_2}(u_1)$, $y_2 = \res_{W(B_4)}^{P_2}(u_2)$, $y_2' = \res_{W(B_4)}^{P_2}(v_2)$, $y_3 = \res_{W(B_4)}^{P_2}(v_2u_1)$ and $y_4 = \res_{W(B_4)}^{P_2}(v_4)$. 

\smallbreak

Now, let $a \in \Inv (P_2, \CM_\ast)$ be any invariant which is induced by an invariant from $\Inv (W(F_4), \CM_\ast)$. Then, we can find unique $m_d \in \CM_{\ast -d}(\gk)$, $m_2, m_2' \in \CM_{\ast -2}(\gk)$ such that 
$$
a = \sum_{\substack{ d \le 4\\d \ne 2}} \Big(\hspace{-.2cm}\sum_{(A, B, C) \in \La^d}\hspace{-.2cm}x_{A, B, C}\Big)m_d + 
 \Big(\hspace{-.2cm}\sum_{(A, B, \es) \in \La^2}\hspace{-.2cm}x_{A, B, \es}\Big)m_2 + \Big(\hspace{-.2cm}\sum_{(\es, \es, C) \in \La^2}\hspace{-.2cm}x_{\es, \es, C}\Big)m_2'.
$$
Now, $s_{1/2(e_1 + e_2 + e_3 + e_4)}$ lies in the normalizer of $P_2$, as it leaves ${a_1}$, ${a_2}$ fixed and swaps ${b_1}$ with $-{b_2}$. Since $a$ comes from $\Inv (W(F_4), \CM_\ast)$, the action of $s_{1/2(e_1 + e_2 + e_3 + e_4)}$ leaves $a$ invariant. Hence,
\begin{align*}
	a = &\sum_{\substack{d \le 4\\d \ne 2}} \Big(\hspace{-.2cm}\sum_{(A, B, C) \in \La^d}\hspace{-.2cm}x_{A, B, C}\Big)m_d
	+ (x_{\{a_1, a_2\}} + x_{\{b_1, b_2\}} + x_{\{a_1, b_1\}} + x_{\{a_2, b_2\}})m_2 \\
	&+ (x_{\{a_1, b_2\}} + x_{\{a_2, b_1\}})m_2'.
\end{align*}
Comparing coefficients yields $m_2 = m_2'$.

\bigbreak

Thus, the image of the restriction $\Inv (W(F_4), \CM_\ast) \to \Inv (P_2, \CM_\ast)$ is contained in the free submodule with basis $\{1, y_1, y_2 + y_2', y_3, y_4\}$. Therefore, the image of the restriction $\Inv (W(F_4), \CM_\ast) \to \Inv (W(B_4), \CM_\ast)$ is contained in the free $\CM_\ast(\gk)$-module with basis $\{1, u_1, v_1, u_2 + v_2, v_1u_1, v_2u_1, v_3, v_4\}$.

\smallbreak

Now, we need to construct $F_4$-invariants which restrict to these elements. First observe that $D_4\ss F_4$ and that $W(F_4)$ stabilizes $D_4$. Thus, any $g \in W(F_4)$ maps the simple system $S = \{e_1 - e_2, e_2 - e_3, e_3 - e_4, e_3 + e_4\}$ to another simple system $S'\ss D_4$. Since all simple systems are conjugate there exists a \emph{unique} $h \in W(D_4)$ mapping $S'$ to $S$. This procedure induces a permutation of the $3$ outer vertices $\{e_1 - e_2, e_3 - e_4, e_3 + e_4\}$ of the Coxeter graph, thereby giving rise to a group homomorphism $\psi:\, W(F_4) \to S_3$.

\smallbreak

Then, we define $v_1:= \psi^*(\wt{w_1})$, where $\wt{w_1} \in \Inv (S_3, \MK_\ast)$ is the first modified Stiefel-Whitney class. To determine the restriction of $v_1$ to $P_L$ note that the map $W(F_4) \to S_3$ sends $W(D_4)$ to the identity and $s_{e_4} $ to the transposition $(2, 3)$. Since $s_{e_i} = g_i s_{e_4} g_i^{-1}$, where $g_i \in W(D_4)$ denotes the element switching the 4th and the $i$th coordinate ($ i \le 3$), we conclude that all $s_{e_i}$ are sent to $(2, 3)$. Thus, the value of $\ms{res}^{P_L}_{W(F_4)}(v_1)$ at the $P_L$-torsor $(\a_1, \b_1, \dots, \a_L, \b_L, \epsilon_{2L + 1}, \dots, \epsilon_4)$ is $\sum_{i \ge 2L + 1} \{\epsilon_i\}$.

\smallbreak

The embedding $W(F_4)\ss O_4$ as orthogonal reflection group yields invariants $\ms{res}^{W(F_4)}_{O_4}(w_d) \in \Inv^d(W(F_4), \MK_\ast)$, where $w_d \in \Inv^d(O_4, \MK_\ast)$ is the $d$th unmodified Stiefel-Whitney class. Again, if $2$ is not a square in $\gk$, then these invariants do not have a nice form, when restricted to the $P_L$. Therefore, we change them a little and define invariants $\widehat{w_d}$. The image of a $P_L$-torsor $(\a_1, \dots, \a_L, \b_1, \dots, \b_L, \epsilon_{2L + 1}, \dots, \epsilon_4)$ in $H^1(k, O_4)$ under the map $P_L\ss W(F_4)\ss O_4$ may be computed by using Example \ref{abQuadratic} and is given by $\lan2\a_1, 2\b_1, \dots, 2\a_L, 2\b_L, \epsilon_{2L + 1}, \dots, \epsilon_4\ran$.
We would like to have
$$
\ms{res}^{P_L}_{W(F_4)}(\wh{w_d}) = \sum_{(A, B, C, E) \in \La^d_L}x^L_{A, B, C, E}.
$$
 Since the restriction of $w_1$ to $P_L$ is already given by $\sum_{(A, B, C, E) \in \La^1_L}x^L_{A, B, C, E}$, we put $\wh w_1:= w_1$. Now, for $d = 2$, 
$$\ms{res}^{P_L}_{O_4}(w_2) = \sum_{(A, B, C, E) \in \La^2_L}x^L_{A, B, C, E} + \sum_{(A, B, \es, \es) \in \La^1_L}\{2\} x^L_{A, B, \es, \es},$$
so that $\wh w_2:= w_2 - \{2\} (w_1 - v_1)$ has the desired property. The restriction of $w_3$ to $P_L$ is 
$$
\ms{res}^{P_L}_{O_4}(w_3) = \sum_{(A, B, C, E) \in \La^3_L}x^L_{A, B, C, E} + \sum_{\substack{(A, B, \es, E) \in \La^2_L\\|E| = 1}}\{2\} x^L_{A, B, \es, E},
$$
so that we set $\wh{w_3}:= w_3-\{2\} (w_1 - v_1) v_1$. Finally, the restriction of $w_4$ to $P_L$ is
$$
\ms{res}^{P_L}_{O_4}(w_4) = \sum_{(A, B, C, E) \in \La^4_L}x^L_{A, B, C, E} + \sum_{\substack{(A, B, C, E) \in \La^3_L\\2|C| + |E| = 2}}\{2\} x^L_{A, B, C, E}
$$
so that we set $\wh{w_4}:= w_4-\{2\}w_2(w_1 - v_1)$. Furthermore, define $u_1:= w_1 - v_1 \in \Inv^1(W(F_4), \MK_\ast)$.

\smallbreak

Now, we restrict the so-constructed invariants to $W(B_4)$. We claim that

\smallbreak

\begin{itemize}
\item[(a)]
$u_1, v_1 \in \Inv^1(W(F_4), \MK_\ast)$ restrict to $u_1, v_1 \in \Inv^1(W(B_4), \MK_\ast)$;

\smallbreak

\item[(b)]
$u_1v_1, (\wh{w_2}-u_1v_1) \in \Inv^2(W(F_4), \MK_\ast)$ restrict to $u_1v_1, u_2 + v_2 \in \Inv^2(W(B_4), \MK_\ast)$; and

\smallbreak

\item[(c)]
$u_1\wh{w_2}, (\wh{w_3}-u_1\wh{w_2}) \in \Inv^3(W(F_4), \MK_\ast)$ restrict to $u_1v_2, v_3$.
\end{itemize}

\smallbreak

Finally, $\wh w_4 \in \Inv^4(W(F_4), \MK_\ast)$ restricts to $v_4 \in \Inv^4(W(B_4), \MK_\ast)$. To prove these claims, we only need to consider the restrictions to $\Inv (P_L, \MK_\ast)$, where the identities are clear by construction. Thus, $\Inv (W(F_4), \CM_\ast)$ is a free $\CM_\ast(\gk)$-module with basis
$$
\{1, \wh{w_1}, v_1, \wh{w_2}, \wh{w_1}v_1, \wh{w_3}, \wh{w_2}v_1, \wh{w_4}\}.
$$
The construction of the $\wh{w_d}$ also yields the following result.
\begin{proposition}
$\Inv (W(F_4), \CM_\ast)$ is completely decomposable with basis
$$
\{1, w_1, v_1, w_2, v_1w_1, w_3, v_1w_2, w_4\}.
$$
\end{proposition}
\begin{remark}
	Alternatively, to the approach above, one could also rely on transfer-restriction arguments to  characterize the invariants of $W(B_4)$, which extend to $W(F_4)$ as those whose restriction to $W(D_4)$ is fixed under the action of $W(F_4)/W(D_4)$.
\end{remark}


\section{Weyl groups of type $D_n$.}
\label{DnSubSect}

\noindent
The root system $D_n$, $n\ge2$ consists of the elements
$$
D_n = \{\pm e_i\pm e_j:\, 1 \le i < j \le n\}.
$$
Let $m:= [n/2]$, $a_i:= e_{2i - 1} - e_{2i}$ and $b_i:= e_{2i - 1} + e_{2i}$. By Remark \ref{splittingPrincipleRem}, this root system defines an orthogonal reflection group over $\gk$ with $|\Omega(W(D_n))| = 1$. More precisely, $P:= P(a_1, b_1, \dots, a_m, b_m)$ is a maximal elementary abelian $2$-group generated by reflections. Furthermore, $W(D_n)$ is a subgroup of $S_n\ltimes (\Z/2)^n\cong W(B_n)$ in the precise sense that
$$
W(D_n) = \{\sigma \cdot \prod_{i \in I}s_{e_i} \in S_n\ltimes (\Z/2)^n:\, |I| \text{ even} \}.
$$
\begin{remark}
	We note that for odd $n$ the invariants of $W(D_n)$ can be deduced from those of $W(B_n)$, since $W(B_n) = \{\pm 1\} \times W(D_n)$. For instance, since $W(D_3) \cong W(A_3)$, this gives the invariants for $W(B_3)$.
\end{remark}

Similarly to the $B_n$-section, we define 
$$\La^d:= \{(A, B, C)\ss [1, m]^3:\, A, B, C \text{ are pw.~disjoint}, \; |A| + |B| + 2|C| = d\}$$
and $x_{A, B, C}:\, H^1(k, P) \to \MK_d(k)/2$
\begin{align*}
 x_{A, B, C}(\a_1, \b_1, \dots, \a_m, \b_m) = \prod_{a \in A}\{\a_a\} \cdot\prod_{b \in B}\{\b_b\} \cdot \prod_{c \in C}\{\a_c\}\{\b_c\}.
\end{align*}
As in the $B_n$-section, we now construct specific invariants. First, for $d \le m$ the group homomorphism $\rho:\, W(D_n)\ss W(B_n) \to S_n$ induces the invariant $u_d:= \rho^*(\wt{w_d}) \in \Inv^d(W(D_n), \MK_\ast)$ with $\ms{res}^P_{W(B_n)}(u_d) = \sum_{\substack{(A, B, \es) \in \La^d}} x_{A, B, \es}.$

\smallbreak

Furthermore, from Section \ref{BnSubSect} we have an embedding $W(D_n)\ss W(B_n)\ss S_{2n}$. Starting with a $W(D_n)$-torsor $x \in H^1(k, W(D_n))$, we may consider its image $q_x \in H^1(k, O_{2n})$ induced by the map $W(D_n) \to S_{2n} \to O_{2n}$. Observe that $W(D_n) \to S_{2n}$ sends
\begin{align*} 
	s_{a_i}\mapsto(2i - 1, 2i)(2i - 1 + n, 2i + n),\;
	s_{b_i}\mapsto (2i - 1, 2i + n)(2i, 2i - 1 + n).
\end{align*} 
Thus, starting with a $P$-torsor $(\a_1, \b_1, \dots, \a_m, \b_m)$, we may apply Lemma \ref{pfisterLemma} to see that under the composition $P \to W(D_n) \to S_{2n} \to O_{2n}$ this torsor maps to
$$
\lan\lan - \a_1, - \b_1 \ran\ran\oplus\dots\oplus\lan\lan - \a_m, - \b_m \ran\ran\;\; (\oplus \lan 1, 1\ran),
$$
where the expression in parentheses appears only for odd $n$. We would like to have an element $v \in \Inv (W(D_n), \MK_\ast)$ such that $\ms{res}^P_{W(D_n)}(v)$ is given by
\begin{align*}
	H^1(k, P)& \to \MK_\ast(k)\\
(\a_1, \b_1 \dots, \a_m, \b_m)&\mapsto (1 + \{\a_1\}\{\b_1\}) \cdots(1 + \{\a_m\}\{\b_m\}).
\end{align*}
To achieve this goal, we proceed recursively as in Section \ref{BnSubSect}. First, we compute the value of the total Stiefel-Whitney class $w \in \Inv (O_4, \MK_\ast)$ at a $2$-fold Pfister form:
\begin{align*}
w(\lan \lan- \a, - \b\ran\ran)& = (1 + \{\a\})(1 + \{\b\})(1 + \{\a\} + \{\b\}) \\
	&= 1 + \{-1\}\{\a\} + \{-1\}\{\b\} + \{\a\}\{\b\}.
\end{align*}
Hence, setting $v' := \ms{res}^{W(D_n)}_{O_{2n}}(w)$, we obtain as in Lemma \ref{bnVdLem} that 
$$\ms{res}^P_{W(D_n)}(v_d') = \hspace{-.2cm} \sum_{(\es, \es, C) \in \La^d} \hspace{-.2cm}x^L_{\es, \es, C} + \sum_{k \le d - 1}\{-1\}^{d - k}\hspace{-.2cm}\sum_{(A, B, C) \in \La^k}\hspace{-.2cm} x_{A, B, C}.$$
Hence, proceeding recursively by setting $v_0 := 0$ and then
$$v_d := v_d' + \sum_{k \le d - 1}u_{d - k} v_k$$
yields the desired invariant. Moreover, 
$
\ms{res}^P_{W(D_n)}(v_d) = \sum_{\substack{(\es, \es, C) \in \La^d}} x_{\es, \es, C}
$
and, by Lemma \ref{bnLemma1},
\begin{align}
	\label{dnProdEq}
\ms{res}^P_{W(D_n)}(u_d) \ms{res}^P_{W(D_n)}(v_e) = \sum_{\substack{(A, B, C) \in \La^{d + e}\\2|C| = e }}x_{A, B, C}.
\end{align}
Now, suppose that $n = 2m$ is even. In this case, we need to construct one further invariant. Since $W(D_n)\cong S_n\ltimes (\Z/2)^{n - 1}$, we have an embedding $S_n\ss W(D_n)$ such that $|W(D_n)/S_n| = 2^{n - 1}$. More precisely, $|W(D_n)/ S_n|$ consists of the cosets $g_IS_n$, where $g_I:=\prod_{i \in I}s_{e_i}$ and where $I\ss [1;n]$ has even cardinality. The left action of $W(D_n)$ on these cosets induces a map $W(D_n) \to S_{2^{n - 1}} \to O_{2^{n - 1}}$. Thus, any $k \in \mc F_{k_0}$ and $y \in H^1(k, W(D_n))$ induce a quadratic form $q_y \in H^1(k, O_{2^{n - 1}})$ and thereby an invariant $\omega \in \Inv (W(D_n), W)$. In fact, we claim that $\omega \in \Inv (W(D_n), I^m)$, where $I(k)\ss W(k)$ is the fundamental ideal. 

\smallbreak

To prove this, we start by showing that $\ms{res}^P_{W(D_n)}(\omega) \in \Inv(P, I^m)$. It is convenient to understand the map $W(D_n) \to S_{2^{n - 1}}$ on the subgroup $P$.
\begin{lemma}
	\label{pactLem}
Let $L = \{\{2i - 1, 2i\}:\, i \le m\}$ and define $f:\, 2^{[1;n]} \to 2^L$,
\begin{align*}
	f(I):= \{\{2i - 1, 2i\}:\, \text{ either }2i - 1 \in I\text{ or }2i \in I\text{, but not both}\}.
\end{align*}
Then, 
\begin{enumerate}
\item 
The action of $P$ on $W(D_n)/S_n$ has the $2^{m - 1}$ orbits $\mc O_{\mc J}:= \{g_I S_n\mid f(I) = {\mc J}\}$, ${\mc J}\ss L$, $|{\mc J}|$ even.

\smallbreak

\item
Let $\mc O_{\mc J}$ be an arbitrary orbit from (1). Put $A_\mc J:= \{i \le m:\, \{2i - 1, 2i\} \in \mc J\}$ and $B_\mc J:= \{i \le m:\, \{2i - 1, 2i\}\not \in \mc J\}$. Then, $P(\{a_i\}_{i \in B_\mc J}\cup\{b_j\}_{j \in A_\mc J})$ acts trivially on $\mc O_{\mc J}$ and the action of $P_{\mc J}:= P(\{a_i\}_{i \in A_\mc J}\cup\{b_j\}_{j \in B_\mc J})$ on $\mc O_{\mc J}$ is simply transitive.
\end{enumerate}
\end{lemma}

\begin{proof}
(1) Let $I\ss [1;n]$. If $\{2i - 1, 2i\}\not \in f(I)$, then $s_{a_i}g_I = g_Is_{a_i}$ and $s_{b_i}g_I = g_{I\Delta\{2i - 1, 2i\}}s_{a_i}$, where $\Delta$ is the symmetric difference. On the other hand, if $\{2i - 1, 2i\} \in f(I)$, then $s_{a_i}g_I = g_{I\Delta\{2i - 1, 2i\}}s_{a_i}$ and $s_{b_i}g_I = g_Is_{a_i}$.

\smallbreak

	(2) By the proof of part (1), $P(\{a_i\}_{i \in B_\mc J}\cup\{b_j\}_{j \in A_\mc J})$ acts trivially on $\mc O_{\mc J}$. Since $|P(\{a_i\}_{i \in A_\mc J}\cup\{b_j\}_{j \in B_\mc J})| = 2^m = |\mc O_{\mc J}|$, assertion (2) follows after verifying that $P(\{a_i\}_{i \in A_\mc J}\cup\{b_j\}_{j \in B_\mc J})$ acts freely on $\mc O_{\mc J}$. So suppose, $I\ss [1;n]$, $M\ss A_\mc J$ and $N\ss B_\mc J$ is such that $f(I) = \mc J$ and $g:= \prod_{i \in M}s_{a_i} \cdot \prod_{j \in N}s_{b_j}$ fixes $g_IS_n$. The proof of part (1) gives that $g g_IS_n = g_{I'}S_n$, where $I' = I\Delta(\cup_{i \in M \cup N}\{2i - 1, 2i\})$. Observing that $I' = I$ if and only if $M = N = \es$ concludes the proof.
\end{proof}

Using Lemma \ref{pactLem}, we conclude the following. Consider an arbitrary $y = (\a_1, \dots, \a_m, \b_1, \dots, \b_m) \in H^1(k, P)$ and let $q_y \in H^1(k, O_{2^{n - 1}})$ be the quadratic form induced by the composition $P \to W(D_n) \to S_{2^{n - 1}} \to O_{2^{n - 1}}$. The decomposition of the action of $P$ into orbits $\mc O_{\mc J}$ induces a decomposition of $q_y$ as $q_y\cong \oplus_{\mc J}q_{\mc J}$. More precisely, the action of $P$ on $\mc O_{\mc J}$ induces a map $P \to S_{2^m}$ and $q_{\mc J}$ is defined to be the image of $y \in H^1(k, P)$ under the composition $P \to S_{2^m} \to O_{2^m}$. By Lemma \ref{pactLem}, this composition factors through the projection $P \to P_{\mc J}$. Now, by Lemma \ref{pfisterLemma}, its remark and Lemma \ref{pactLem},
\begin{align}
	\label{qjEq}
q_{\mc J}\cong \lan 2^m\ran\otimes \bigotimes_{i \in A_{\mc J}}\lan\lan - \a_i\ran\ran\otimes \bigotimes_{j \in B_{\mc J}}\lan\lan - \b_j\ran\ran.
\end{align}
Thus, the image of $q_y = \oplus_{\mc J}q_{\mc J}$ in $W(k)$ lies in $I^m(k)$, so that $\ms{res}^P_{W(D_n)}(\omega) \in \Inv(P, I^m)$.

\smallbreak

Now, we pass from $P$ to $W(D_n)$. First, $\omega$ induces an invariant $\ol{\omega} \in \Inv^0(W(D_n), I^*/I^{* + 1})$ through the projection $W \to (I^*/I^{* + 1})_0 = W/I$. Since the image of $\ms{res}^P_{W(D_n)}(\omega)$ lies in $I^m\ss I$, we conclude that $\ms{res}^P_{W(D_n)}(\ol{\omega}) = 0$. As $P$ is up to conjugation the only maximal elementary abelian $2$-subgroup of $W(D_n)$ generated by reflections, Corollary \ref{splitCorollary} gives that $\ol\omega = 0 \in \Inv^0(W(D_n), I^*/I^{* + 1})$, i.e., $\omega \in \Inv(W(D_n), I)$. Iterating this procedure $m$ times shows that $\omega \in \Inv(W(D_n), I^m)$. 

\smallbreak

By Example \ref{pfisterInvariant}, there exists an invariant $e_m:\, I^m(k) \to \MK_2(k)$ satisfying 
\begin{align}
	\label{emyEq}
e_m(\lan\lan\a_1\ran\ran\otimes \cdots\otimes\lan\lan\a_m\ran\ran) = \prod_{i \le m}\{\a_i\}.
\end{align}
Then, 
$$e_m(y):= e_m(\lan2^m\ran\otimes\omega(y)) + \{-1\}\sum_{k \le d - 1}u_{d - 1 - k} v_k$$ 
defines an element of $\Inv^m(W(D_n), \MK_\ast)$ and, in the vein of Lemma \ref{bnVdLem}, we now determine its restriction to $P$.
\begin{lemma}
	\label{dnEmLem}
\begin{equation}
\label{dnEm}
	\ms{res}^P_{W(D_n)}(e_m) = \sum_{\substack{(A, B, \es) \in \La^m\\ |A|\text{ even}}}x_{A, B, \es}
\end{equation}
\end{lemma}
\begin{proof}
	First, by identity \eqref{dnProdEq}, it suffices to show that the restriction of the invariant $e_m'(y) :=  e_m(\lan2^m\ran\otimes\omega(y))$ to $P$ is given by 
	\begin{align}
		\label{dnEmp}
		\sum_{\substack{(A, B, \es) \in \La^m\\ |A|\text{ even}}}x_{A, B, \es} + \{-1\}\sum_{\substack{(A, B, C) \in \La^m_{d - 1}}}x_{A, B, C}.
	\end{align}
	Then, by identities \eqref{qjEq} and \eqref{emyEq}, evaluating $\ms{res}^P_{W(D_n)}(e_m')$ at the torsor $(\a_1, \dots, \a_m, \b_1, \dots, \b_m) \in H^1(k, P)$ gives that 
	\begin{align*}
		\sum_{\substack{(A, B, \es) \in \La^m\\ |A|\text{ even}}}\prod_{i \in A}\{-\a_i\}\prod_{j \in B}\{-\b_j\} 
		&=\hspace{-.6cm} \sum_{\substack{(A, B, \es) \in \La^m\\ |A|\text{ even}}}\sum_{\substack{U \ss A\\ V \ss B}}\{-1\}^{m - |U| - |V|}\prod_{i \in U}\{\a_i\}\hspace{-.1cm}\prod_{j \in V}\{\b_j\}\\
		&= \hspace{-.5cm}\sum_{\substack{U, V \ss [1, m]\\ U \cap V = \es}}\hspace{-.1cm} N_{U, V}\{-1\}^{m - |U| - |V|}\prod_{i \in U}\{\a_i\}\hspace{-.0cm}\prod_{j \in V}\{\b_j\}.
	\end{align*}
	where 
	$$N_{U, V} := |\{A \ss [1,m]:\, A\supset U , A \cap V = \es , |A| \text{ even}\}|.$$
	To conclude the proof, we distinguish on the value of $|U| + |V|$. First, the contributions coming from $|U| +|V| = m$ give precisely the leading-order expression in \eqref{dnEmp}.
	Next, suppose that $|U| + |V| = m - k$ with $k \ge 1$. Then,  $N_{U, V} = 2^{k - 1}$, so that the corresponding contribution vanishes mod 2 if and only if $k \ge 1$. Now, we conclude the proof by noting that the contributions for $k = 1$ yield precisely the summation expression in \eqref{dnEmp}.
\end{proof}

\smallbreak

Now, we derive a central set of constraints for the image of the restriction map $\Inv(W(D_n), \CM_\ast) \to \Inv(P, \CM_\ast)$. For $d \le n$ and $i \le [d/2]$ put
$$
\phi^d_i:= \sum_{\substack{(A, B, C) \in \La^d\\|C| = i}}x_{A, B, C} \in \Inv^d(P, \MK_\ast)
$$
and
$
\psi_1:= \sum_{\substack{(A, B, \es)\\ |A|
\text{ even}}}x_{A, B, \es}.
$

\begin{lemma}
	\label{pactLem2}
The image of the restriction map $\Inv(W(D_n), \CM_\ast) \to \Inv(P, \CM_\ast)$ is contained in the free $\CM_\ast(\gk)$-module with basis 
$$
S = \{\phi^d_i:\, d \le n, \; \max(0, d - m) \le i \le [d/2]\}\cup R, 
$$
where $R = \es$, if $n$ is odd and $R = \{\psi_1\}$, if $n$ is even.
\end{lemma}

\begin{proof}
	Arguing as in the $B_n$-section shows that all elements of $S$ are non-zero. Furthermore, both
$s_{e_{2i - 1} - e_{2j-1}} s_{e_{2i} - e_{2j}}$ and $s_{e_{2i - 1}} s_{e_{2j-1}}$ normalize $P$.

	Let us denote by $N_1, N_2\ss N(P)$ the subgroups generated by the first, respectively second kind of elements and let us denote by $N$ the subgroup generated by $N_1$ and $N_2$. At the torsor level, conjugation by the first kind of elements swaps $\a_i \lra \a_j$ and $\b_i \lra \b_j$. Thus for $(A, B, C) \in \La^d$, the invariant $x_{A, B, C}$ maps to $x_{A', B', C'}$, where $A' = (i, j)A$, $B' = (i, j)B$ and $C' = (i, j)C$. On the other hand, conjugation by the second kind of elements swaps $\a_i \lra \b_i$ and $\a_j \lra \b_j$. Thus, it maps $x_{A, B, C}$ to $x_{A', B', C}$, where $A' = (A - \{i, j\})\cup (B \cap \{i, j\})$ and $B' = (B - \{i, j\})\cup (A \cap\{i, j\})$. That is, if $i \in A$, we remove it from $A$ and put it into $B$ and vice versa; then we do the same for $j$. Thus, $N$ acts on $\Inv (P, \MK_\ast)$ by permuting the $x_{A, B, C}$ and hence we can apply Corollary \ref{orbitSum}. 
	
	In the next step, we determine the orbit of $x_{A_0, B_0, C_0}$ under $N$ for an arbitrary $(A_0, B_0, C_0) \in \La^d$. First, suppose that $n$ is odd or that $C_0 \ne \es$ or that ($n = 2m$ is even and $d < m$). Then, we claim that the orbit of $x_{A_0, B_0, C_0}$ under $N_2$ is given by $\{x_{A, B, C_0}:\,(A, B, C_0) \in \La^d, \;A\cup B = A_0\cup B_0\}$. It suffices to show that for any $a \in A_0$, there exists an element of $N_2$ mapping $x_{A_0, B_0, C_0}$ to $x_{A_0-\{a\}, B_0\cup\{a\}, C_0}$. As soon as this is proven, one observes that the symmetric statement with $b \in B_0$ also holds; iterating these operations, we indeed get the claimed orbit. For $n$ odd, $s_{e_{2a-1}} s_{e_n}$ maps $x_{A_0, B_0, C_0}$ to $x_{A_0-\{a\}, B_0\cup\{a\}, C_0}$. If $C_0 \ne \es$ choose $c \in C_0$; then $s_{e_{2a-1}} s_{e_{2c-1}}$ maps $x_{A_0, B_0, C_0}$ to $x_{A_0-\{a\}, B_0\cup\{a\}, C_0}$. Finally, if $n = 2m$ is even and $d < m$, then there exists $i \in [1;m]$ such that $i\not \in A_0\cup B_0\cup C_0$ and the element $s_{e_{2a-1}} s_{e_{2i - 1}}$ does the trick. Thus, the orbit of $x_{A_0, B_0, C_0}$ under $N_2$ equals $\{x_{A, B, C_0}:\,(A, B, C_0) \in \La^d, \;A\cup B = A_0\cup B_0\}$. Similarly, for any $(A_1, B_1, C_1) \in \La^d$ the orbit of $x_{A_1, B_1, C_1}$ under $N_1$ equals $\{x_{A, B, C}:\,(A, B, C) \in \La^d, \;|A| = |A_1|, \;|B| = |B_1|, \;|C| = |C_1|\}$. Combining these results, the orbit of $x_{A_0, B_0, C_0}$ under $N$ is given by $\{x_{A, B, C}:\,(A, B, C) \in \La^d, \;|C| = |C_0|\}$. 

\smallbreak

Finally, let $C_0 = \es$, $n = 2m$ be even and $d = m$. Then, the orbit of $x_{A_0, B_0, \es}$ under $N_2$ equals $\{x_{A, B, \es}:\,(A, B, \es) \in \La^d, \;A\cup B = A_0\cup B_0, |B| - |B_0|\;\text{is even}\}$. Using that for any $(A_1, B_1, C_1) \in \La^d$ the orbit of $x_{A_1, B_1, C_1}$ under $N_1$ is given by $\{x_{A, B, C}:\, (A, B, C) \in \La^d, \;|A| = |A_1|, \;|B| = |B_1|, \;|C| = |C_1|\}$, we see that the orbit of $x_{A_0, B_0, \es}$ under $N$ is $\{x_{A, B, \es}:\, (A, B, \es) \in \La^d, \;|B| -|B_0|\;\text{is even}\}$.

\smallbreak

Hence, applying Corollary \ref{orbitSum} concludes the proof.
\end{proof}

In particular, as Lemma \ref{bnLemma1} gives that $\ms{res}^P_{W(D_n)}(u_{d - 2i}v_{2i}) = \phi^d_i$ and as $\ms{res}^P_{W(D_n)}(e_m) = \psi_1$ and , we obtain the following result.
\begin{corollary}
	\label{dnCor}
$\Inv (W(D_n), \CM_\ast)$ is completely decomposable with basis
$$
\{u_{d - 2i}v_{2i}:\, d \le n, \max(0, d - m) \le i \le [d/2]\}\cup R, 
$$
where $R = \es$ for odd $n$ and $R = \{e_m\}$ for even $n$.
\end{corollary}
\begin{remark}
	A relation between $W(B_n)$ and $W(D_n)$ explains why in Corollary \ref{dnCor}, we only see $v_d$ with even $d$. Indeed, the kernel of the determinant of the $2n$-dimensional representation of $W(B_n)$ contains $W(D_n)$. Since for odd $d$, all the $W(B_n)$-invariants $v_d$ are divisible by $v_1$ and since $v_1$ is vanishing, we deduce that they all reduce to 0 on $W(D_n)$.
\end{remark}


\goodbreak
\section{Weyl groups of type $E_6$, $E_7$, and $E_8$.}
\label{E6-8SubSect}

\subsection{Type $E_6$}
Up to conjugacy, $P:= P(a_1, b_1, a_2, b_2)$ is the unique maximal elementary abelian subgroup generated by reflections in $W(E_7)$. Since the injection $\Inv (W(E_6), \CM_\ast) \to \Inv (P, \CM_\ast)$ factors through $\Inv (W(D_5), \CM_\ast)$, the map $\Inv (W(E_6), \CM_\ast) \to \Inv (W(D_5), \CM_\ast)$ is injective and a basis of $\Inv (W(D_5), \CM_\ast)$ is given by $\{1, u_1, u_2, v_2, v_2u_1, v_4\}$.

\smallbreak

So let $a \in \Inv (P, \CM_\ast)$ be an invariant which comes from a $W(E_6)$-invariant. Since the inclusion
$P\ss W(E_6)$ factors through $W(D_5)\ss W(E_6)$, $a$ decomposes uniquely as
$$
a = \sum_{\substack{ d \le 4\\d \ne 2}}\sum_{(A, B, C) \in \La^d}x_{A, B, C}m_d + 
\sum_{(A, B, \es) \in \La^2}x_{A, B, \es}m_2 + \sum_{(\es, \es, C) \in \La^2}x_{\es, \es, C}m_2'
$$
for certain $m_d \in \CM_{\ast -d}(\gk)$, $m_2, m_2' \in \CM_{\ast -2}(\gk)$. Now, the element
$$
g:= s_{\frac12(e_1 - e_2 - e_3 - e_4 - e_5 - e_6 - e_7 + e_8)} s_{\frac12( - e_1 + e_2 + e_3 + e_4 - e_5 - e_6 - e_7 + e_8)} \in W(E_6)
$$
lies in the normalizer of $P$, since 
\begin{align*}
g s_{a_1} g^{-1} = s_{b_2}, \qquad
g s_{b_1} g^{-1} = s_{b_1}, \qquad
g s_{a_2} g^{-1} = s_{a_2}, \qquad
g s_{b_2} g^{-1} = s_{a_1}.
\end{align*}
The induced action of $g$ on a $P$-torsor $(\a_1, \a_2, \b_1, \b_2)$ is thus given by swapping $\a_1 \lra \b_2$, while leaving $\a_2, \b_1$ fixed. Therefore, applying $g$ to the invariant $a$ yields
\begin{align*}
\sum_{\substack{ d \le 4\\d \ne 2}} \sum_{(A, B, C) \in \La^d}x_{A, B, C}m_d
	+ \sum_{i,j \in \{1, 2\}}x_{\{a_i, b_j\}} m_2 + (x_{\{a_1, a_2\}} + x_{\{b_1, b_2\}})m_2'.
\end{align*}
Since $a$ comes from an invariant of $W(E_6)$, it stays invariant under $g$ and comparing coefficients, we conclude that the image of the restriction $\Inv (W(E_6), \CM_\ast) \to \Inv (W(D_5), \CM_\ast)$ lies in the free submodule with basis 
$$\{1, u_1, u_2 + v_2, v_2u_1, v_4\}.$$
The embedding of $W(E_6)$ in $O_8$ as orthogonal reflection group gives rise to the invariants $\ms{res}^{W(E_6)}_{O_8}(\wt{w_d}) \in \Inv^d(O_8, \MK_\ast)$, which we again denote by $\wt{w_d}$. For any $k \in \Fields_{\gk}$ and $(\a_1, \b_1, \a_2, \b_2) \in (k^\times/k^{\times 2})^4$, the map $P \to W(E_6)\ss O_8$ induces the quadratic form
$$
\lan2\a_1, 2\b_1, 2\a_2, 2\b_2, 1, 1, 1, 1\ran.
$$
Thus, the total modified Stiefel-Whitney class evaluated at this torsor equals 
$$
(1 + \{\a_1\})(1 + \{\a_2\})(1 + \{\b_1\})(1 + \{\b_2\}).
$$
Now,
\begin{align*}
\ms{res}^P_{W(D_5)}(u_1)& = \ms{res}^P_{W(E_6)}(\wt{w_1}),
	&\ms{res}^P_{W(D_5)}(u_2 + v_2) &= \ms{res}^P_{W(E_6)}(\wt{w_2}),\\
	\ms{res}^P_{W(D_5)}(v_2u_1)& = \ms{res}^P_{W(E_6)}(\wt{w_3}),
	&\ms{res}^P_{W(D_5)}(v_4) &= \ms{res}^P_{W(E_6)}(\wt{w_4}).
\end{align*}
Hence, $\{\wt{w_d}\}_{d \le 4}$ form a basis of $\Inv(W(E_6), \CM_\ast)$ as $\CM_\ast(\gk)$-module.

\medbreak

\subsection{Type $E_7$}
Up to conjugacy, $P:= P(a_1, b_1, a_2, b_2, a_3, b_3, a_4)$ is the unique maximal elementary abelian subgroup generated by reflections in $W(E_7)$. Looking at the root systems, we see that there is an inclusion $W(D_6)\times \lan s_{a_4}\ran\ss W(E_7)$.
Invoking the same factorization argument as before, the restriction map
$$
\Inv (W(E_7), \CM_\ast) \to \Inv (W(D_6)\times\lan s_{a_4}\ran, \CM_\ast )
$$
is injective. We first recall that $\Inv (W(D_6)\times \lan s_{a_4}\ran, \CM_\ast)$ is a free $\CM_\ast(\gk)$-module with basis
\begin{enumerate}
\item[(0)] 1

\smallbreak

\item $u_1, x_{\{a_4\}}$

\smallbreak

\item $u_2, v_2, u_1x_{\{a_4\}}$

\smallbreak

\item $(u_3 - e_3), e_3, u_1v_2, u_2x_{\{a_4\}}, v_2x_{\{a_4\}}$

\smallbreak

\item $u_2v_2, v_4, (u_3 - e_3)x_{\{a_4\}}, e_3x_{\{a_4\}}, u_1v_2x_{\{a_4\}}$

\smallbreak

\item $v_4u_1, u_2v_2x_{\{a_4\}}, v_4x_{\{a_4\}}$

\smallbreak

\item $v_6, v_4u_1x_{\{a_4\}}$

\smallbreak

\item $v_6x_{\{a_4\}}.$
\end{enumerate}
Defining $g:= s_{\frac12(e_1 - e_2 - e_3 - e_4 - e_5 - e_6 - e_7 + e_8)} s_{\frac12( - e_1 + e_2 + e_3 + e_4 - e_5 - e_6 - e_7 + e_8)} \in W(E_7)$ as in the $E_6$-case yields that
\begin{align*}
g s_{a_1} g^{-1}& = s_{b_2},
	&g s_{b_1} g^{-1}& = s_{b_1},
	&g s_{a_2} g^{-1}& = s_{a_2},
	&g s_{b_2} g^{-1}& = s_{a_1},\\
	g s_{a_3} g^{-1} &= s_{a_3},
	&g s_{b_3} g^{-1}& = s_{a_4},
	&g s_{a_4} g^{-1}& = s_{b_3}.
\end{align*}
The action of $g$ on a $P$-torsor $(\a_1, \b_1, \dots, \a_3, \b_3, \a_4) \in (k^\times/k^{\times 2})^7$ is thus given by swapping $\a_1 \lra \b_2$, $\b_3 \lra\a_4$ while leaving $\b_1, \a_2, \a_3$ fixed. Arguing just as in the $E_6$-case, we see that the image of $\Inv (W(E_7), \CM_\ast) \to \Inv (W(D_6)\times\lan s_{a_4}\ran, \CM_\ast)$ lies in the free $\CM_\ast(\gk)$-module with basis 
\begin{enumerate}
\item[(0)] 1

\smallbreak

\item $u_1 + x_{\{a_4\}}$

\smallbreak

\item $v_2 + u_2 + u_1x_{\{a_4\}}$

\smallbreak

\item $u_1v_2 + (u_3 - e_3) + u_2x_{\{a_4\}}, e_3 + v_2x_{\{a_4\}}$

\smallbreak

\item $v_4 + (u_3 - e_3)x_{\{a_4\}}, u_2v_2 + u_1v_2x_{\{a_4\}} + e_3x_{\{a_4\}}$

\smallbreak

\item $v_4x_{\{a_4\}} + u_2v_2x_{\{a_4\}} + v_4u_1$

\smallbreak

\item $v_4u_1x_{\{a_4\}} + v_6$

\smallbreak

\item $v_6x_{\{a_4\}}$.
\end{enumerate}
Now, we provide specific $W(E_7)$-invariants. First, the embedding $W(E_7)\ss O_8$ gives us invariants $\ms{res}^{W(E_7)}_{O_8}(\wt{w_d}) \in \Inv^d(W(E_7), \MK_\ast)$, which we again denote by $\wt{w_d}$. Then,
\begin{align*}
\ms{res}^P_{W(E_7)}(\wt{w_1})& = \ms{res}^P_{W(D_6)\times\lan s_{{a_4}}\ran}(u_1 + x_{\{a_4\}})\\
\ms{res}^P_{W(E_7)}(\wt{w_2})& = \ms{res}^P_{W(D_6)\times\lan s_{{a_4}}\ran}(u_2 + v_2 + u_1x_{\{a_4\}})\\
\ms{res}^P_{W(E_7)}(\wt{w_3})& = \ms{res}^P_{W(D_6)\times\lan s_{{a_4}}\ran}(u_3 + u_1v_2 + u_2x_{\{a_4\}} + v_2x_{\{a_4\}})\\
\ms{res}^P_{W(E_7)}(\wt{w_4})& = \ms{res}^P_{W(D_6)\times\lan s_{{a_4}}\ran}(u_2v_2 + v_4 + u_3x_{\{a_4\}} + u_1v_2x_{\{a_4\}})\\
\ms{res}^P_{W(E_7)}(\wt{w_5})& = \ms{res}^P_{W(D_6)\times\lan s_{{a_4}}\ran}(v_4u_1 + v_4x_{\{a_4\}} + u_2v_2x_{\{a_4\}})\\
\ms{res}^P_{W(E_7)}(\wt{w_6})& = \ms{res}^P_{W(D_6)\times\lan s_{{a_4}}\ran}(v_6 + v_4u_1x_{\{a_4\}})\\
\ms{res}^P_{W(E_7)}(\wt{w_7})& = \ms{res}^P_{W(D_6)\times\lan s_{{a_4}}\ran}(v_6x_{\{a_4\}}).
\end{align*}
So we still lack invariants in degree $3$ and $4$. To construct the missing invariant in degree $3$, we mimic the construction of the invariant ${e_m}$ in the $D_n$-section. Let $U\cong S_6\times \lan s_{a_4}\ran \ss W(E_7)$ be the subgroup generated by the reflections at
$$
\{e_1 + e_2, e_2 - e_3, e_3 - e_4, e_4 - e_5, e_5 - e_6, e_7 - e_8\}.
$$
Then, $|U\backslash W(E_7)| = 2016$
and we obtain a map $W(E_7) \to S_{2016} \to O_{2016}$. To be more precise, there is a right action of $W(E_7)$ on the right cosets $U\backslash W(E_7)$ given by right multiplication. This induces an anti-homomorphism $W(E_7) \to S_{2016}$ and precomposing this map with $g\mapsto g^{-1}$, we obtain the desired homomorphism. We need the following lemma which tells us that we are in a situation which is quite similar to the $D_n$-case:

\begin{lemma}
\label{e7GapLem}
Let $k \in \mc F_{\gk}$ and $y \in H^1(k, P)$ be a $P$-torsor. Let $q_y$ be the quadratic form induced by $y$ under the composition $P \to W(E_7) \to S_{2016} \to O_{2016}$. Then, the image of $q_y$ in $W(k)$ is contained in $I^3(k)$.
\end{lemma}

\begin{proof}
This can be checked by a computational algebra system, see the appendix.
\end{proof}

\medbreak

\noindent
We now argue similarly to the $D_n$-case.
In concrete terms, if $y$ is a $W(E_7)$-torsor, and $q_y$ is the quadratic form induced by $y$ under the composition $W(E_7) \to S_{2016} \to O_{2016}$, then the image of $q_y$ in $W(k)$ is contained in $I^3(k)$ and we define the invariant 
\begin{align}
	\label{f3pEq}
	f_3'(y):= e_3(\lan 2^3\ran\otimes q_y).
\end{align}
In the $D_n$-case, namely in Lemma \ref{dnEmLem}, we could compute the restriction of the invariant $e_m$ to the maximal elementary abelian $2$-subgroup explicitly. In principle, this would also be possible in the present setting. However, the computations would be substantially more involved. Therefore, we provide a more conceptual level argument. To that end, we recall from Section \ref{DnSubSect} that if $g \in W(E_7)$ is contained in the normalizer $N_{W(E_7)}(P)$ of $P$ in $W(E_7)$, then $g$ acts both on the invariants $\{x_{A, B, C}\}_{(A, B, C)\in \La^d} \in \Inv^d (P, \CM_\ast)$ as well as on the indexing set $\La^d$.
\begin{lemma}
	\label{orbitLem}
	Let $d \le 7$ and $g \in N_{W(E_7)}(P)$. Also, let $a \in \Inv^d(W(E_7), \MK_\ast)$ be an invariant and represent its restriction to $\Inv^d(P, \MK_\ast)$ as
	\begin{align}
		\label{orbitEq}
		\ms{res}^P_{W(E_7)}(a) = \sum_{\ell \le d} \sum_{I \in \La^{\ell}}m_Ix_I,
	\end{align}
	for certain coefficients $m_I \in \MK_{d - |I|}(\gk)$. Then, $m_I = m_{g(I)}$ for all $\ell \le d$ and $I \in \La^\ell$.
\end{lemma}
\begin{proof}
	First, since the restriction is invariant under the action of $g$, 
	\begin{align}
		\label{orbitEq2}
		 \sum_{\ell \le d} \sum_{I \in \La^{d - \ell}}(m_I - m_{g(I)})x_I = 0.
	\end{align}
	Now, suppose that the assertion of the lemma was false, and choose a counterexample  $I^* \in \La^{\ell^*}$ with maximal $\ell^*$. Then, we first evaluate both sides of \eqref{orbitEq} at the function field $E = k_0(A_1, B_1, \dots, A_3, B_3, A_4)$ in the indeterminates $A_1, B_1, \dots, A_3, B_3, A_4$ corresponding to the roots in $P$, and then apply the Milnor residue maps corresponding to the indeterminates associated with the index set $I^*$. Since $\ell^*$ was chosen to be maximal, the identity \eqref{orbitEq2} reduces to $m_I - m_{g(I)} = 0$, which concludes the proof.
\end{proof}
In words, just as in Corollary \ref{orbitSum}, when representing the restrictions of invariants as in \eqref{orbitEq}, then basis elements in the same orbit share the same coefficient. 

In particular, we have seen above that in degree 1 and 2 all basis elements are in a single orbit and are therefore the restriction of the corresponding modified Stiefel-Whitney classes. Thus, applying Lemma \ref{orbitLem} with $a = f_3'$, there exist $m_\ell \in \MK_{3 - \ell}(\gk)$, $\ell \in \{0, 1, 2\}$ and $m_{A, B, C} \in \Z/2$, $(A, B, C) \in \La^3$ such that
$$\ms{res}^P_{W(E_7)}(f_3') = \sum_{(A, B, C) \in \La^3}m_{A, B, C}x_{A, B, C} + \sum_{\ell \le 2}m_\ell \ms{res}^P_{W(E_7)}(\wt{w_\ell}).$$
Then, proceeding as in the definition of $e_m$ in Section \ref{DnSubSect}, we define an invariant $f_3 \in \Inv^3(W(E_7), \MK_\ast)$ by stripping of the mixed terms from $f_3'$. That is,
$$f_3 := f_3' - \sum_{\ell \le 2}m_\ell \wt{w_\ell}.$$
In the appendix, we expound on how a computational algebra system shows that
\begin{align}
	\label{e7BarEq}
 \ms{res}^P_{W(E_7)}(f_3)= \ms{res}^P_{W(D_6)\times\lan s_{a4}\ran}(u_1v_2 + u_3 - e_3 + u_2x_{\{a_4\}}).
\end{align}
Finally, we can proceed in a similar fashion in order to remove the mixed terms in the product expression.
\begin{align*}
&(u_1 + x_{\{a_4\}})(u_1v_2 + (u_3 - e_3) + u_2x_{\{a_4\}}).
\end{align*}
Thus, $\Inv (W(E_7), \CM_\ast)$ is completely decomposable with basis $\{\wt{w_d}\}_{d \le 7}\cup\{f_3, f_3 \wt{w_1}\}$.

\medbreak

\subsection{Type $E_8$}
Up to conjugacy, $P:= P(a_1, b_1, a_2, b_2, a_3, b_3, a_4, b_4)$ is the unique maximal elementary abelian subgroup generated by reflections in $W(E_8)$. By the same arguments as in the $E_6/E_7$-case, we obtain that the restriction map $\Inv (W(E_8), \CM_\ast) \to \Inv (W(D_8), \CM_\ast)$ is injective. We first recall that $\Inv (W(D_8), \CM_\ast)$ is a free $\CM_\ast(\gk)$-module with the basis
$$\{1, u_1, u_2, v_2, u_3, v_2u_1,e_4, v_4, (u_4 - e_4), v_2u_2, v_2u_3, v_4u_1, v_4u_2, v_6, v_6u_1, v_8\}.$$
Again, we define $g \in W(E_8)$ as in the $E_6$ or $E_7$-case and check that it normalizes $P$:
\begin{align*}
	g s_{a_1} g^{-1}& = s_{b_2},&
	g s_{b_1} g^{-1} = s_{b_1}, \qquad &
	g s_{a_2} g^{-1} = s_{a_2},&
g s_{b_2} g^{-1} = s_{a_1},\\
	g s_{a_3} g^{-1}& = s_{a_3},&
	g s_{b_3} g^{-1} = s_{a_4}, \qquad &
	g s_{a_4} g^{-1} = s_{b_3}, &
g s_{b_4} g^{-1} = s_{b_4}.
\end{align*}
The action of $g$ on a $P$-torsor $(\a_1, \b_1, \a_2, \b_2, \a_3, \b_3, \a_4, \b_4)$ is thus given by swapping $\a_1 \lra \b_2$, $\b_3 \lra\a_4$ while leaving $\b_1, \a_2, \a_3, \b_4$ fixed. Again, applying the same kind of arguments as in the $E_6$-case, we see that the image of the restriction map $\Inv (W(E_8), \CM_\ast) \to \Inv (W(D_8), \CM_\ast)$ is contained in the free submodule with basis
$$\{1, u_1, u_2 + v_2, u_3+ v_2u_1,e_4 + v_4, (u_4 - e_4) + v_2u_2, v_2u_3 + v_4u_1, v_4u_2 + v_6, v_6u_1, v_8\}.$$
We need to construct $W(E_8)$-invariants mapping to these basis elements. On the one hand, the inclusion $W(E_8)\ss O_8$ gives modified Stiefel-Whitney classes $\wt{w_d} \in \Inv^d(W(E_8), \MK_\ast)$. Again,
\begin{alignat*}3
	\res^P_{W(E_8)}(\wt{w_1})& = \res^P_{W(D_8)}(u_1),&&
	\hspace{-1.2cm}\res^P_{W(E_8)}(\wt{w_5}) = \res^P_{W(D_8)}(v_2u_3 + v_4u_1),\\
	\res^P_{W(E_8)}(\wt{w_2}) &= \res^P_{W(D_8)}(u_2 + v_2),&&
	\hspace{-.8cm}\res^P_{W(E_8)}(\wt{w_6}) = \res^P_{W(D_8)}(v_4u_2 + v_6),\\
	\res^P_{W(E_8)}(\wt{w_3}) &= \res^P_{W(D_8)}(u_3 + u_1v_2)	,&&
	\res^P_{W(E_8)}(\wt{w_7}) = \res^P_{W(D_8)}(v_6u_1),\\
	\res^P_{W(E_8)}(\wt{w_4}) &= \res^P_{W(D_8)}(u_4 + u_2v_2 + v_4),\quad&&
	\hspace{.4cm}\res^P_{W(E_8)}(\wt{w_8}) = \res^P_{W(D_8)}(v_8).
\end{alignat*}
The situation is very similar to the $E_7$-case except that now, we miss a basis invariant in degree $4$. Let $U\ss W(E_8)$ be the subgroup generated by the reflections at
$$
\{e_1 + e_2, e_2 - e_3, e_3 - e_4, e_4 - e_5, e_5 - e_6, e_6 - e_7, e_7 - e_8\}.
$$
By observing that $U\cong S_8$ or by using a computational algebra software, we conclude $|U\backslash W(E_8)| = 17280$. As in the $E_7$-case, we obtain a map $W(E_8) \to S_{17280} \to O_{17280}$. Again, we need the following lemma.
\begin{lemma}
\label{e8GapLem}
Let $k \in \mc F_{k_0}$ and $y \in H^1(k, P)$ be a $P$-torsor. Let $q_y$ be the quadratic form induced by $y$ under the composition $P \to W(E_8) \to S_{17280} \to O_{17280}$. Then, the image of $q_y$ in $W(k)$ is contained in $I^4(k)$.
\end{lemma}

\begin{proof}
Again, this can be checked by a computational algebra software, see the appendix.
\end{proof}

\smallbreak

\noindent
As in the $D_n$-case, we obtain from this an invariant $f_4 \in \Inv^4(W(E_8), \MK_\ast)$. More precisely, if $y$ is a $W(E_8)$-torsor and $q_y$ is the quadratic form induced by $y$ under the composition $W(E_8) \to S_{17280} \to O_{17280}$, then the image of $q_y$ in $W(k)$ is contained in $I^4(k)$ and we define $f_4'(y):= e_4(q_y)$. We then proceed as in the $E_7$-case and set
$$f_4 := f_4' - \sum_{\ell \le 3}m_\ell \wt{w_\ell}$$
for suitable $m_\ell \in \MK_\ell(4 - \ell)$ in order to strip off the mixed contributions from $f_4'$.

\smallbreak

The restriction of $f_4$ to $P$ is determined through a computational algebra system, see the appendix. The result is 
$
\ms{res}^P_{W(D_8)}(v_2u_2 + (u_4 - e_4)).
$
Thus, we conclude that $\Inv (W(E_8), \CM_\ast)$ is completely decomposable with basis
$
\{f_4\}\cup\{\wt{w_d}\}_{d \le 8}.
$

\medbreak

\section{Appendix A -- Excerpts from a letter by J.-P.~Serre}
\label{serreSec}

[...] Hence, the only technical point which remains is the ``splitting principle'': if the restrictions of an invariant to every cube is 0, the invariant is 0. In your text with Gille, you prove that result under the restrictive condition that the characteristic $p$ does not divide the order $|G|$ of the group $G$. The proof you give (which is basically the same as in my UCLA lectures) is based on the fact that the polynomial invariants of $G$ (in its natural representation) make up a polynomial algebra; in geometric language, the quotient ${\rm Aff}^n/G$ is isomorphic to ${\rm Aff}^n$. This is OK when $p$ does not divide $|G|$, but it is also true in many other cases. For instance, it is true for all $p$ $(\ne 2)$ for the classical types (provided, for type $A_n$, that we choose for lattice the natural lattice for $GL_{n+1}$, namely $\mathbb Z^{n+1}$). For types $G_2, F_4, E_6, E_7$, it is true if $p > 3$ and for $E_8$ it is true for $p > 5$: this is not easy to prove, but it has been known to topologists since the 1950's (because the question is related to the determination of the mod $p$ cohomology of the corresponding compact Lie groups). When I started working on these questions, I found natural to have to exclude, for instance, the characteristics 3 and 5 for $E_8$. It is only a few years ago that I realized that even these small restrictions are unnecessary: the splitting principle holds for every $p > 2$.

I have sketched the proof in my Oberwolfach report: take for instance the case of $E_8$; the group $G = W(E_8)$ contains $W(D_8)$ as a subgroup of odd index, namely 135; moreover, the reflections of $W(D_8)$ are also reflections of $W(E_8)$; hence every cube of $W(D_8)$ is a cube of $W(E_8)$; if a cohomological invariant of $W(E_8)$ gives 0 over every cube, its restriction to $W(D_8)$ has the same property, hence is 0 because $D_8$ is a classical type; since the index of $W(D_8)$ is odd, then this invariant is 0. It is remarkable that a similar proof works in every other case. [...]

\section{Appendix B -- Computations for $E_7$ and $E_8$}
\label{apdxSec}

\noindent
For the computations involving $E_7$ and $E_8$, we use the computational algebra system {\tt GAP} and the {\tt GAP}-package {\tt CHEVIE} \cite{CH}. The complete source code used for the proof of Lemmas \ref{e7GapLem} and \ref{e8GapLem} together with detailed instructions on how to reproduce the results are provided on the author's GitHub page:\, {\tt https://github.com/Christian-Hirsch/orbit-e78}.

\subsection{Computations concerning $W(E_7)$}
\label{e7Apdx}

The proof of Lemma \ref{e7GapLem} requires detailed information on the action of $P$ on $U\backslash W(E_7)$. We analyze this action, via the procedure {\tt fullCheck(7, U, P)}. 

First, {\tt fullCheck(7, U, P)} computes the action of $P$ on $U\backslash W(E_7)$ and also its orbits $\Oc_1, \dots, \Oc_r$.
Then, for each orbit $\Oc_k$, it determines a subset $A_k \ss \{a_1, b_1, a_2, b_2, a_3, b_3, a_4\}$, such that
$P(\{a_1, b_1, a_2, b_2, a_3, b_3, a_4\}-A_k)$ acts trivially on $\Oc_k$ and such that $P(A_k)$ acts simply transitively
on $\Oc_k$. A priori, there is no reason that such a subset should exist; however -- as checked by the program --
it exists in the case we are considering. The return value of the procedure {\tt fullCheck} is an array whose $k$th entry is the set
$A_k$. Inspecting the return value reveals that each $A_k$ consists of at least 3 elements and that the subsets consisting of 3 elements have the desired form.

More precisely, to call {\tt fullCheck(7, U, P)}, we need to determine the indices of the roots generating $U$ and $P$.
In the following, the roots are expressed as linear combinations of the simple system of roots given by
$
v_1 = \tfrac12(e_1 - e_2 - e_3 - e_4 - e_5 - e_6 - e_7 + e_8)$, $v_2 = e_1 + e_2$, $v_i = e_{i - 1} - e_{i - 2}$, $3 \le i \le 7$.
Additionally, 
\begin{align*}
b_2& = v_2 + v_3 + 2v_4 + v_5\\
b_3& = v_2 + v_3 + 2v_4 + 2v_5 + 2v_6 + v_7\\
-a_4& = 2v_1 + 2v_2 + 3v_3 + 4v_4 + 3v_5 + 2v_6 + v_7
\end{align*}
We claim that $U$ and $P$ are represented by the indices $[2, 4, 5, 6, 7, 63]$ and $[3, 2, 5, 28, 7, 49, 63]$, respectively. This can be checked by printing the basis representation of the $E_7$ roots:

\noindent gap$>$ p: = [ 3, 2, 5, 28, 7, 49, 63 ]; \newline
gap$>$ for u in p do Print(CoxeterGroup("E", 7).roots[u]);Print("$\backslash$ n");od; \newline
[ 0, 0, 1, 0, 0, 0, 0 ] \newline
[ 0, 1, 0, 0, 0, 0, 0 ] \newline
[ 0, 0, 0, 0, 1, 0, 0 ] \newline
[ 0, 1, 1, 2, 1, 0, 0 ] \newline
[ 0, 0, 0, 0, 0, 0, 1 ] \newline
[ 0, 1, 1, 2, 2, 2, 1 ] \newline
[ 2, 2, 3, 4, 3, 2, 1 ] \newline

We can now call the {\tt fullCheck}-procedure.

\noindent gap$>$ Aks: = fullCheck(7, [2, 4, 5, 6, 7, 63], [3, 2, 5, 28, 7, 49, 63]); \newline

Verifying that all $\{A_k\}_{k \le r}$ consist of at least 3 elements can be achieved via the command

\noindent gap$>$ for Ak in Aks do if Length(Ak)$ < $3 then Print("Fail");fi;od; \newline

To see that those $A_k$ with $|A_k| = 3$ correspond precisely to the elements
 \begin{align*}
	 \{(A, B, C) \in \La_3:\, |C| = 1\} &\cup \{(A, B, \es) \in \La_3:\, |A|\text{ odd}\} \\
	 &\cup \{(A, B, \es, a_4):\, (A, B, \es) \in \La_2\},
 \end{align*}
we use the {\tt e7Correct}-procedure. It checks that the $\{A_k\}_{k \le r}$ do not contain elements which
are not in the claimed set above. Since there are precisely 28 $A_k$ with 3 elements, which is precisely the cardinality of the above set, this reasoning yields the claimed description.

\noindent gap$>$ Y: = Filtered(Aks, Ak-$>$ Length(Ak)$ < $4);
 \newline
\noindent gap$>$ e7Correct(Y);

\subsection{Computations concerning $W(E_8)$}
\label{e8Apdx}

Since the arguments are very similar to the $E_7$-case, we only explain the most important changes. First, we consider the maximal elementary abelian subgroup generated by reflections $P = P(a_1, b_1, a_2, b_2, a_3, b_3, a_4, b_4)$ and the subgroup
$$U = \lan s_{e_1 + e_2}, s_{e_2 - e_3}, s_{e_3 - e_4}, s_{e_4 - e_5}, s_{e_5 - e_6}, s_{e_6 - e_7}, s_{e_7 - e_8}\ran.$$ 
In addition to the computations provided in Appendix \ref{e7Apdx}, we note that
$$b_4 = 2v_1 + 3v_2 + 4v_3 + 6v_4 + 5v_5 + 4v_6 + 3v_7 + 2v_8.$$
Then, $P$ and $U$ are represented by the indices $[3, 2, 5, 32, 7, 61, 97, 120]$ and $[2, 4, 5, 6, 7, 8, 97]$:

\noindent gap$>$ a: = [3, 2, 5, 32, 7, 61, 97, 120]; \newline
[ 3, 2, 5, 32, 7, 61, 97, 120 ] \newline
gap$>$ for u in a do Print(CoxeterGroup("E", 8).roots[u]); Print("$\backslash$ n"); od; \newline
[ 0, 0, 1, 0, 0, 0, 0, 0 ] \newline
[ 0, 1, 0, 0, 0, 0, 0, 0 ] \newline
[ 0, 0, 0, 0, 1, 0, 0, 0 ] \newline
[ 0, 1, 1, 2, 1, 0, 0, 0 ] \newline
[ 0, 0, 0, 0, 0, 0, 1, 0 ] \newline
[ 0, 1, 1, 2, 2, 2, 1, 0 ] \newline
[ 2, 2, 3, 4, 3, 2, 1, 0 ] \newline
[ 2, 3, 4, 6, 5, 4, 3, 2 ] \newline

To understand the orbit structure, we proceed as in the $E_7$-case:

\noindent gap$>$ Aks: = fullCheck(8, [2, 4, 5, 6, 7, 8, 97], [3, 2, 5, 32, 7, 61, 97, 120]); \newline
\noindent gap$>$ for Ak in Aks do if Length(Ak)$ < $4 then Print("Fail");fi;od; \newline
\noindent gap$>$ Y: = Filtered(Aks, Ak-$>$Length(Ak)$ < $5); \newline
\noindent gap$>$ e8Correct(Y);

\bibliographystyle{amsa}

\addresseshere

\end{document}